\documentclass[12pt,reqno]{amsart}
\usepackage[margin=1in]{geometry}
\usepackage{amsmath,amssymb,amsthm,graphicx,amsxtra, setspace}
\usepackage[utf8]{inputenc}
\usepackage{mathrsfs}
\usepackage{hyperref}
\usepackage{upgreek}
\usepackage{mathtools}
\usepackage[mathcal]{euscript}
\usepackage{xcolor}
\usepackage{multirow}
\allowdisplaybreaks

\usepackage[pagewise]{lineno}

\newtheorem{theorem}{Theorem}[section]
\newtheorem{lemma}[theorem]{Lemma}
\newtheorem{proposition}[theorem]{Proposition}

\newtheorem{definition}[theorem]{Definition}

\newtheorem{remark}[theorem]{Remark}

\newtheorem{hypothesis}[theorem]{Hypothesis}

\let\originalleft\left
\let\originalright\right
\renewcommand{\left}{\mathopen{}\mathclose\bgroup\originalleft}
\renewcommand{\right}{\aftergroup\egroup\originalright}

\newcommand{\Tr}{\mathop{\mathrm{Tr}}}
\renewcommand{\d}{\/\mathrm{d}\/}

\def\w{\textbf{W}^{\varepsilon}_{{\theta}^{\varepsilon}}}

\def\L{\mathbb{L}}
\def\A{\mathrm{A}}
\def\I{\mathrm{I}}
\def\F{\mathrm{F}}
\def\C{\mathrm{C}}
\def\f{\boldsymbol{f}}

\def\B{\mathrm{B}}
\def\D{\mathrm{D}}

\def\X{\mathbb{X}}
\def\x{\boldsymbol{x}}

\def\h{\boldsymbol{h}}
\def\z{\boldsymbol{z} }
\def\v{\boldsymbol{v}}
\def\w{\boldsymbol{w}}
\def\W{\mathrm{W}}

\def\N{\mathbb{N}}

\def\V{\mathbb{V}}

\def\u{\mathrm{U}}
\def\P{\mathrm{P}}
\def\u{\boldsymbol{u}}
\def\H{\mathbb{H}}

\newcommand{\R}{\mathbb{R}}

\renewcommand{\d}{\/\mathrm{d}\/}

\newcommand{\Addresses}{{
		\footnote{
			\noindent \textsuperscript{1}School of Mathematical Sciences, Guizhou Normal University, Guiyang 550001, P.R. China.

			\noindent \textsuperscript{2,3}Department of Mathematics, Indian Institute of Technology Roorkee-IIT Roorkee,
			Haridwar Highway, Roorkee, Uttarakhand 247667, INDIA.
			\par\nopagebreak
			\noindent  \textit{e-mail:} \texttt{Manil T. Mohan: \href{maniltmohan@ma.iitr.ac.in}{maniltmohan@ma.iitr.ac.in}, \href{maniltmohan@gmail.com}{maniltmohan@gmail.com}.}
			
			\textit{e-mail:} \texttt{Renhai Wang: \href{rwang-math@outlook.com}{rwang-math@outlook.com}.}
			
			\textit{e-mail:} \texttt{Kush Kinra: \href{kkinra@ma.iitr.ac.in}{kkinra@ma.iitr.ac.in}.}
			
			\noindent \textsuperscript{*}Corresponding author.
			
			\textit{Key words:} Pullback random attractor, Asymptotic autonomy, stochastic Navier-Stokes equations, backward tempered set, time-semi-uniform asymptotic compactness, backward uniform-tail estimate, backward flattening estimate.
			
			Mathematics Subject Classification (2020): Primary 37L55; Secondary 37B55, 35B41, 35B40.

}}}

\begin{document}
	
	\title[Asymptotically autonomous robustness of random attractors for 2D SNSE]{Asymptotically autonomous robustness in probability of random attractors for  stochastic Navier-Stokes equations on unbounded Poincar\'e domains
		\Addresses}
	
	\author[R. Wang, K. Kinra  and M. T. Mohan]
	{Renhai Wang\textsuperscript{1}, Kush Kinra\textsuperscript{2} and Manil T. Mohan\textsuperscript{3*}}

	\maketitle
	
	\begin{abstract}
The asymptotically autonomous robustness of random attractors of stochastic fluid equations defined on \emph{bounded} domains has been considered in the literature. In this article,  we initially consider this topic (almost surely and in probability) for a non-autonomous stochastic 2D Navier-Stokes equation driven by additive and multiplicative noise defined on some \emph{unbounded Poincar\'e domains}. There are two significant keys to study this topic: what is the asymptotically autonomous limiting set of the time-section of random attractors as time goes to negative infinity, and how to show the precompactness of a time-union of random attractors over an \emph{infinite} time-interval $(-\infty,\tau]$. We guess and prove that such a limiting set is just determined by the random attractor of a stochastic Navier-Stokes equation driven by an autonomous forcing satisfying a convergent condition. The  uniform ``tail-smallness'' and ``flattening effecting'' of the solutions are derived in order to justify that the usual asymptotically compactness of the solution operators is \emph{uniform} over $(-\infty,\tau]$. This in fact leads to the precompactness of the time-union of random attractors over $(-\infty,\tau]$. The idea of uniform tail-estimates due to Wang \cite{UTE-Wang} is employed to overcome the noncompactness of Sobolev embeddings on unbounded domains. Several rigorous calculations are given to deal with the pressure terms when we derive these uniform tail-estimates.
	\end{abstract}
	

	\section{Introduction} \label{sec1}\setcounter{equation}{0}
\subsection{Statement of problems}
The well-posedness and
global/pullback/exponential /trajectory/random attractors of 2D Navier-Stokes
equations defined on bounded domains have been well-discussed in the literature, see  \cite{Ball1997JNS,FSX,FMRT,GZ,MS,SS,R.Temam} and many others. The existence of deterministic and random attractors of a Navier-Stokes equation defined on the whole space $\mathbb{R}^2$ is an interesting and challenging open problem. Motivated by several interesting works such as Caraballo, Lukaszewicz and Real \cite{CLR,CLR1}, Brz\'ezniak et.al. \cite{BCLLLR,BL}, Gu, Guo and Wang \cite{GGW},
 and others, one can, however, try these kind of analysis on unbounded Poincar\'e domains. By a Poincar\'e  domain, we mean a domain in which the Poincar\'e inequality is satisfied.  A typical example of unbounded Poincar\'e domains in $\mathbb{R}^2$ is $\mathcal{O}=\R\times(-L,L)$ with $L>0$, see Temam \cite[p.306]{R.Temam} and Robinson \cite[p.117]{Robinson2}.

		\begin{hypothesis}\label{assumpO}
			Let $\mathcal{O}$ be an open, connected and unbounded subset of $\R^2$, the boundary of which is uniformly of class $\mathrm{C}^3$ (see \cite{Heywood}). We  assume that, there exists a positive constant $\lambda $ such that the following Poincar\'e inequality  is satisfied:
			\begin{align}\label{poin}
				\lambda\int_{\mathcal{O}} |\psi(x)|^2 \d x \leq \int_{\mathcal{O}} |\nabla \psi(x)|^2 \d x,  \ \text{ for all } \  \psi \in \H^{1}_0 (\mathcal{O}).
			\end{align}
		\end{hypothesis}
	
This work is devoted to the study of asymptotically autonomous robustness of random attractors
for a 2D non-autonomous stochastic Navier-Stokes
fluid defined on the unbounded Poincar\'e domain $\mathcal{O}$:
		\begin{equation}\label{1}
			\left\{
			\begin{aligned}
				\frac{\partial \u}{\partial t}-\nu \Delta\u+(\u\cdot\nabla)\u+\nabla p&=\boldsymbol{f}+S(\u)\circ\frac{\d \W}{\d t}, \ \ \  \text{ in }\  \mathcal{O}\times(\tau,\infty), \\ \nabla\cdot\u&=0, \hspace{30mm}  \text{ in } \ \ \mathcal{O}\times(\tau,\infty), \\ \u&=0,\hspace{30mm}  \text{ in } \ \ \partial\mathcal{O}\times(\tau,\infty), \\
				\u(x,\tau)&=\u_0(x),\hspace{23mm} x\in \mathcal{O} \ \text{ and }\ \tau\in\R,
			\end{aligned}
			\right.
		\end{equation}
		where $\tau\in\mathbb{R}$, $\u(x,t) \in \R^2$, $p(x,t)\in\R$ and $\f(x,t)\in \R^2$ denotes the velocity field, pressure and external forcing, respectively, the positive constant $\nu$ is known as the \emph{kinematic viscosity} of the fluid, the term $S(\u)$ is referred as the \textbf{diffusion  coefficient} of the noise, and it is either independent of $\u$, that is, $S(\u)=\h\in\D(\A)$ (additive noise) or equal to $\u$ (multiplicative noise), the symbol $\circ$ means that  the stochastic integral is understood in the sense of Stratonovich, $\W=\W(t,\omega)$ is an one-dimensional two-sided Wiener process defined on a standard  probability space $(\Omega, \mathscr{F}, \mathbb{P})$, and $\D(\A)$ is the domain of the Stokes operator.

By the asymptotically autonomous robustness of random attractors, we mean that the time-section of random attractors is robust (or stable) in terms of the Hausdorff semi-distance of the energy space as the time-parameter goes to negative infinity.

\subsection{Literature survey}
The theories and applications of global/pullback/exponential /trajectory/ attractors of deterministic dynamical systems can be referred to some prominent works of Hale and Raugel \cite{Hale1,Hale2}, Temam \cite{R.Temam}, Robinson \cite{Robinson2,Robinson1}, Ball \cite{Ball}, Chueshov and Lasiecka, \cite{Chueshov1,Chueshov2}, Chepyzhov and Vishik  \cite{CV2}, Caraballo, Lukaszewicz and Real \cite{CLR,CLR1}, Carvalho, Langa and Robinson \cite{CLR2} and many others. In order to capture the long-time dynamics of stochastic equations driven  by uncertain forcing, attractors of deterministic dynamical systems were extended to be random attractors of random dynamical systems, see Arnold \cite{Arnold}, Brze\'zniak, Capi\'nski and Flandoli \cite{BCF}, Crauel, Debussche and Flandoli \cite{CDF,CF} and Schmalfu{\ss} \cite{Schmalfussr}. Since evolution equations arriving from physics and other fields of science are often driven by stochastic and non-autonomous forcing simultaneously, random attractors of autonomous random dynamical systems are generalized by Wang \cite{SandN_Wang} under the framework of non-autonomous random dynamical systems. In light of these theoretical results, random attractor have been extensively investigated in \cite{BCLLLR,FY,GGW,GLW,KM3,KM7,
LG,PeriodicWang,rwang1,rwang2}, etc. for autonomous and non-autonomous stochastic equations.  As it is well known in the literature that  to study pathwise attractors, one needs to convert a stochastic system into a pathwise deterministic one, which is possible for additive or linear multiplicative noise. Very recently, a new concept called mean random attractors of mean random dynamical system was proposed by Wang \cite{Wang} in order to study the long-term behavior of solutions of stochastic It\^{o} evolution equations driven by nonlinear noise. We mention that the existence of mean random attractors of stochastic It\^{o} Navier-Stokes equations driven by  nonlinear noise has been studied by Wang \cite{Wang1}, see \cite{Chenzhang2,KM4,Wang4,Wang8,Wang9}, etc. for other stochastic models.

\subsection{Motivations, conditions and  main results}		
It is well known that the time dependence of forcing term reflects the non-autonomous feature of evolution systems. This could be the most important feature distinguishing from autonomous evolution systems. Intuitively, if the non-autonomous forcing term $\f(x,t)$ in \eqref{1} converges asymptotically to an autonomous forcing term in some sense, then the non-autonomous dynamics of \eqref{1} become more and more autonomous. In such a  case we call this phenomenon asymptotically autonomous dynamics of \eqref{1}. Our main motivation is to investigate the asymptotically autonomous robustness of random attractors of  \eqref{1} driven by additive and multiplicative noise when $\f$ and  $\h$ satisfy the following conditions:
		\begin{hypothesis}\label{Hypo_f-N}
 $\f\in\mathrm{L}^{2}_{\emph{loc}}
(\R;\L^2(\mathcal{O}))$ converges to a time-independent function $\f_{\infty}\in\L^2(\mathcal{O})$:
			\begin{align*}
				\lim_{\tau\to -\infty}\int^{\tau}_{-\infty}
\|\f(t)-\f_{\infty}\|^2_{\L^2(\mathcal{O})}\d t=0.
			\end{align*}
	\end{hypothesis}

\begin{hypothesis}\label{AonH-N}
There exists a constant ${\aleph}>0$ such that
$\h\in\D(\A)$ satisfies
	\begin{align*}
		\bigg|
\sum_{i,j=1}^2\int_{\mathcal{O}}\u_i(x)
\frac{\partial \h_j(x)}{\partial x_i}\u_j(x)\d x\bigg|\leq {\aleph}\|\u\|^2_{\mathbb{L}^2(\mathcal{O})}, \ \ \forall\ \u\in\L^2(\mathcal{O}).
	\end{align*}
\end{hypothesis}

\begin{remark}
	Hypothesis \ref{Hypo_f-N} implies the following conditions (see Caraballo et al. \cite{CGTW}):
	\begin{itemize}
		\item Uniformness condition:
		\begin{align}\label{G3}
			&\sup_{s\leq \tau}\int_{-\infty}^{s}e^{\kappa(r-s)} \|\f(t)\|^2_{\L^2(\mathcal{O})}\d r<+\infty,  \ \mbox{$\forall\ \kappa>0$, $\tau\in\mathbb{R}$}.
		\end{align}
		\item The tails  of the forcing $\f$ are backward-uniformly small:
		\begin{align}\label{f3-N}
			\lim_{k\rightarrow\infty}\sup_{s\leq \tau}\int_{-\infty}^{s}e^{\kappa(r-s)}
			\int_{\mathcal{O}\cap\{|x|\geq k\}}|\f(x,r)|^{2}\d x\d r=0,  \ \mbox{$\forall\  \kappa>0$, $\tau\in\mathbb{R}$.}
		\end{align}
	\end{itemize}
\end{remark}

\begin{remark}
\begin{itemize}	\item [(i)] An example of Hypothesis \ref{Hypo_f-N} is
	$\f(x,t)=\f_\infty(x)e^t+\f_\infty(x)$ with $\f_{\infty}\in\L^2(\mathcal{O})$.
	\item[(ii)] An example of $\h=(h_1,h_2)$ in Hypothesis \ref{AonH-N} is any $\h\in\D(\A)$  with
	$\sup\limits_{x\in\mathcal{O}}\left|
	\frac{\partial h_j(x)}{\partial x_i}\right|<\infty,$
	$i,j=1,2$.
	\item[(iii)] Hypothesis \ref{AonH-N} is only used for
	additive noise case, where $S(\u)=\h\in\D(\A)$.
	\item[(iv)] Hypothesis \ref{Hypo_f-N} is used to prove that the solutions of \eqref{1} is
	asymptotically autonomous
	in $\L^2(\mathcal{O})$.
\end{itemize}
\end{remark}

With these conditions, we are able to state our main results, which reveal the asymptotically autonomous dynamics of \eqref{1}.
\begin{theorem}\label{MT1-N}
(\texttt{Additive noise case})
Under Hypotheses \ref{assumpO}-\ref{AonH-N}, the non-autonomous random dynamical system $\Phi$ generated by \eqref{1} with $S(\u)=\h$ has a
unique pullback random attractor
$\mathcal{A}
    =\{\mathcal{A}(\tau,\omega):\tau\in\mathbb{R},
    \omega\in\Omega\}$ such that
$\bigcup_{s\in(-\infty,\tau]}\mathcal{A}
(s,\omega)$ is precompact in $\L^2(\mathcal{O})$ and  $$\lim_{t \to +\infty} e^{- \gamma t}\sup_{s\in(-\infty,\tau]
}\|\mathcal{A}(s-t,\theta_{-t} \omega ) \|_{\L^2(\mathcal{O})} =0,$$ for any
$\gamma>0$, $\tau\in \mathbb{R}$ and  $\omega\in\Omega$.  In addition, the time-section  $\mathcal{A}(\tau,\omega)$ is asymptotically
autonomous robust in $\L^2(\mathcal{O})$, and the limiting set of $\mathcal{A}(\tau,\omega)$ as $\tau\rightarrow-\infty$ is just determined
by the random attractor $\mathcal{A}_{\infty}=
\{\mathcal{A}(\omega):
    \omega\in\Omega\}$ of a stochastic Navier-Stokes equation \eqref{A-SNSE} with the autonomous forcing $\f_\infty$, that is,
	\begin{align}\label{MT2-N}
		\lim_{\tau\to -\infty}\mathrm{dist}_{\L^2(\mathcal{O})}
(\mathcal{A}(\tau,\omega),
\mathcal{A}_{\infty}(\omega))=0, \ \ \forall\ \omega\in\Omega.
	\end{align}
Furthermore, we also justify the
\texttt{asymptotically autonomous robustness in probability}:
\begin{align}\label{MT3-N}
\lim_{\tau\to -\infty}\mathbb{P}\Big(\omega\in\Omega:\mathrm{dist}_{\L^2(\mathcal{O})}
(\mathcal{A}(\tau,\omega),
\mathcal{A}_{\infty}(\omega))\geq\delta\Big){=0},\ \ \ \forall\ \delta>0.
	\end{align}
\end{theorem}
\begin{theorem}\label{MT1}(\texttt{Multiplicative noise case})
Under Hypotheses \ref{assumpO} and \ref{Hypo_f-N}, all results in  Theorem \ref{MT1-N} hold for
the non-autonomous random dynamical system generated by \eqref{1} with $S(\u)=\u$.
\end{theorem}

\subsection{Novelties, difficulties and approaches}
A crucial point to  prove  \eqref{MT2-N} and \eqref{MT3-N} is how to obtain the uniform
 precompactness of
$\bigcup_{s\in(-\infty,\tau]}\mathcal{A}
(s,\omega)$ in $\L^2(\mathcal{O})$.
It is known from the theoretical results of Wang \cite{SandN_Wang} that the pullback asymptotic compactness of $\Phi$ implies the compactness of every single time-section $\mathcal{A}
(\tau,\omega)$. Since $(-\infty,\tau]$ is an infinite interval,
one cannot expect that the usual pullback asymptotical compactness of $\Phi$ leads to the
precompactness of
$\bigcup_{s\in(-\infty,\tau]}\mathcal{A}
(s,\omega)$ in $\L^2(\mathcal{O})$. However, motivated by the ideas of \cite{SandN_Wang}, this can be done if one is able to show that the usual pullback asymptotically compactness of $\Phi$ is uniform with respect to a uniformly tempered universe (see \eqref{D-NSE})
over $(-\infty,\tau]$.

When  $\mathcal{O}$ is a bounded domain, the \emph{uniform} pullback asymptotically compactness of $\Phi$ over $(-\infty,\tau]$ has been established in \cite{KRM} via a compact  uniform pullback absorbing set by using compact Sobolev embeddings. This argument has been widely used to study similar topics for other types of   stochastic fluid equations defined on bounded domains, see Wang and Li \cite{WL} for MHD equations; and \cite{LX,LX1,XL,YGL,ZL}, etc. for $g$-Navier-Stokes equations and Brinkman-Forchheimer equations.

When  $\mathcal{O}$ is a unbounded domain as considered as an unsolved problem in the present paper, the Sobolev embeddings are no longer compact, proving such \emph{uniform} asymptotically compactness is therefore hard than the bounded domain case. In this paper, we use the idea of uniform tail-estimates due to Wang \cite{UTE-Wang} to  overcome the noncompactness of Sobolev embeddings on unbounded domains, that is, we will use a cut-off technique to prove that solutions to  \eqref{1} are sufficiently small in $\L^2(\mathcal{O}^c_k)$ uniformly over $(-\infty,\tau],$ when $k$ is large enough, where $\mathcal{O}_{k}=\{x\in\mathcal{O}:|x|\leq k\}$ and $\mathcal{O}_{k}=\mathbb{R}^2\setminus \mathcal{O}_{k}$. Essentially unlike the parabolic or hyperbolic equations as considered in \cite{chenp,CGTW,LGL,UTE-Wang,rwang1},  etc. the fluid equations like \eqref{1} contain the pressure term $p$. When we derive these uniform tail-estimates, the pressure term $p$ can not simply treated by the divergence theorem (does not vanish). However, by taking the divergence in \eqref{1} and using the incompressibility condition, we get  the  rigorous expression of the pressure term $p=(-\Delta)^{-1}\big[\sum_{i,j=1}^{2} \frac{\partial^2}{\partial x_i\partial x_j}(u_iu_j)\big]=(-\Delta)^{-1}[\Tr(\nabla u)^2]$. Then it is possible to derive these uniform tail-estimates, but we shall carefully deal with the $\L^4(\mathcal{O})$-norm of the solutions resulting from the expression of $p$, see \eqref{ep7-N}. As a  result of these uniform tail-estimates and uniform ``flattening effecting'' of solutions to \eqref{1}, the uniform pullback asymptotical compactness of $\Phi$ in $\L^2(\mathcal{O})$ follows. It is worth mentioning that the wide-spread idea of energy equations developed by Ball \cite{Ball} can be used to overcome the noncompactness of Sobolev embeddings on unbounded domains, see \cite{BCLLLR,BL,CLR,CLR1,GGW,KM7,PeriodicWang, Wang2011Tran,wangjfa,rwang2} and many others. We remark that we are currently unable to use the idea of energy equations to prove the uniform pullback asymptotical compactness of $\Phi$ in $\L^2(\mathcal{O})$ since $(-\infty,\tau]$ is an infinite time-interval.

Since we have to consider the uniformly tempered universe to prove the uniform pullback asymptotically compactness of $\Phi$, we shall prove the measurability of the uniformly compact  attractor. This is not straightforward compared with the the usual case since the radii of the uniform pullback absorbing set is taken as the supremum over an uncountable set $(-\infty,\tau]$ (see Proposition \ref{IRAS-N}). In order to  surmount the difficulty, we first observe that the measurability of the usual random attractor is known in the literature, see for example, \cite{BCLLLR,BL,CLR,CLR1,GGW}, and then prove that such a uniformly compact attractor is just equal to the usual random attractor. This idea has been successfully used by Caraballo et.al. \cite{CGTW} and Wang and Li \cite{WL} for different stochastic models.

\subsection{Outline}
In the next section, we consider an abstract formulation of \eqref{1}, discuss  the properties of an Ornstein-Uhlenbeck process and Kuratowski's measure of noncompactness. In Section \ref{sec3}, we prove Theorem \ref{MT1-N} for problem \eqref{1} driven by additive noise. In the final section, we prove Theorem \ref{MT1} for problem \eqref{1} driven by multiplicative noise.

		\section{Mathematical Formulations and Preparations}\label{sec2}\setcounter{equation}{0}
	In this section, we discuss the necessary function spaces needed to obtain the results of this work. Next, we define linear and bilinear operators which help us to obtain an abstract formulation of the stochastic system \eqref{1}. Further, we discuss the Ornstein-Uhlenbeck process, its properties and the backward tempered random sets. Finally, we discuss Kuratowski's measure of noncompactness with its consequence (Lemma \ref{K-BAC}). Note that Lemma \ref{K-BAC} plays a crucial role to prove the time-semi-uniform asymptotic compactness (see Subsection \ref{thm1.4}).
	\subsection{Function spaces and operators}
	Let the space $\mathcal{V}:=\{\u\in\C_0^{\infty}(\mathcal{O};\R^2):\nabla\cdot\u=0\},$ where $\C_0^{\infty}(\mathcal{O};\R^2)$ denote the space of all infinite times differentiable functions  ($\R^2$-valued) with compact support in $\mathcal{O}$. Let $\H$ and $\V$ denote the completion of $\mathcal{V}$ in 	$\mathrm{L}^2(\mathcal{O};\R^2)$ and $\mathrm{H}_0^1(\mathcal{O};\R^2)$ norms, respectively. The spaces  $\H$ and $\V$ are endowed with the norms $\|\u\|_{\H}^2:=\int_{\mathcal{O}}|\u(x)|^2\d x$ and $\|\u\|_{\V}^2:=\int_{\mathcal{O}}|\nabla\u(x)|^2\d x$ (using Poincar\'e's inequality),  respectively. The induced duality between the spaces $\V$ and $\V^*$ is denoted by $\langle\cdot,\cdot\rangle.$ Moreover, we have the continuous embedding $\V\hookrightarrow\H\equiv\H^*\hookrightarrow\V^*.$
	
		\subsubsection{Linear operator}\label{LO}
	Let $\mathcal{P}: \L^2(\mathcal{O}) \to\H$ denote the Helmholtz-Hodge orthogonal projection (cf.  \cite{OAL}). Let us define the Stokes operator
	\begin{equation*}
		\A\u:=-\mathcal{P}\Delta\u,\;\u\in\D(\A).
	\end{equation*}
	The operator $\A:\V\to\V^*$ is a linear continuous operator.	Since the boundary of $\mathcal{O}$ is uniformly of class $\mathrm{C}^3$, it is inferred that $\D(\A)=\V\cap\H^2(\mathcal{O})$ and $\|\A\u\|_{\H}$ defines a norm in $\D(\A),$ which is equivalent to the one in $\H^2(\mathcal{O})$ (cf. Lemmas  1, \cite{Heywood}). The above argument implies that $\mathcal{P}:\H^2(\mathcal{O})\to\H^2(\mathcal{O})$ is a bounded operator. Moreover, the operator $\A$ is  non-negative self-adjoint  in $\H$ and
	\begin{align}\label{2.7a}
		\langle\A\u,\u\rangle =\|\u\|_{\V}^2,\ \textrm{ for all }\ \u\in\V, \ \text{ and }\ \|\A\u\|_{\V^*}\leq \|\u\|_{\V}.
	\end{align}
	\subsubsection{Bilinear operator}\label{BO}
	Let us define the \emph{trilinear form} $b(\cdot,\cdot,\cdot):\V\times\V\times\V\to\R$ by $$b(\u,\v,\w)=\int_{\mathcal{O}}(\u(x)\cdot\nabla)\v(x)\cdot\w(x)\d x=\sum_{i,j=1}^2\int_{\mathcal{O}}\u_i(x)
\frac{\partial \v_j(x)}{\partial x_i}\w_j(x)\d x.$$ If $\u, \v$ are such that the linear map $b(\u, \v, \cdot) $ is continuous on $\V$, the corresponding element of $\V^*$ is denoted by $\B(\u, \v)$. We also denote $\B(\u) = \B(\u, \u)=\mathcal{P}[(\u\cdot\nabla)\u]$.	An integration by parts gives
	\begin{equation}\label{b0}
		\left\{
		\begin{aligned}
			b(\u,\v,\v) &= 0,\ \text{ for all }\ \u,\v \in\V,\\
			b(\u,\v,\w) &=  -b(\u,\w,\v),\ \text{ for all }\ \u,\v,\w\in \V.
		\end{aligned}
		\right.\end{equation}
	\begin{remark}
	1. 	The following inequality is used in the sequel $($see Chapter 2, section 2.3, \cite{Temam1}$):$
			\begin{align}
				|b(\u,\v,\w)|&\leq
				C\|\u\|^{1/2}_{\H}\|\u\|^{1/2}_{\V}
\|\v\|_{\V}\|\w\|^{1/2}_{\H}\|\w\|^{1/2}_{\V},\ \ \text{ for all } \u, \v, \w\in \V.  \label{b1}
			\end{align}

	2. 	Note that $\langle\B(\u,\u-\v),\u-\v\rangle=0$, which  implies that
		\begin{equation}\label{441}
			\begin{aligned}
				\langle \B(\u)-\B(\v),\u-\v\rangle =\langle\B(\u-\v,\v),\u-\v\rangle=-\langle\B(\u-\v,\u-\v),\v\rangle.
			\end{aligned}
		\end{equation}
	\end{remark}

\subsection{Abstract formulation and Ornstein-Uhlenbeck process}\label{2.5}
Taking the projection $\mathcal{P}$ on the 2D SNSE equations \eqref{1}, one obtains
\begin{equation}\label{SNSE}
	\left\{
	\begin{aligned}
		\frac{\d\u(t)}{\d t}+\nu \A\u(t)+\B(\u(t))&=\f(t) +S(\u(t))\circ\frac{\d \W(t)}{\d t} , \\
		\u(x,\tau)&=\u_{0}(x),\ \ \	x\in \mathcal{O},
	\end{aligned}
	\right.
\end{equation}
where $S(\u)=\u$ or independent of $\u$ (for simplicity of notations, we denoted $\mathcal{P}S(\u)$ as $S(\u)$ and $\mathcal{P}\f$ as $\f$). Here, $\W(t,\omega)$ is a  standard scalar Wiener process on the probability space $(\Omega, \mathscr{F}, \mathbb{P}),$ where $$\Omega=\{\omega\in C(\R;\R):\omega(0)=0\},$$ endowed with the compact-open topology given by the complete metric
\begin{align*}
	d_{\Omega}(\omega,\omega'):=\sum_{m=1}^{\infty} \frac{1}{2^m}\frac{\|\omega-\omega'\|_{m}}{1+\|\omega-\omega'\|_{m}},\ \text{ where }\  \|\omega-\omega'\|_{m}:=\sup_{-m\leq t\leq m} |\omega(t)-\omega'(t)|,
\end{align*}
and $\mathscr{F}$ is the Borel sigma-algebra induced by the compact-open topology of $(\Omega,d_{\Omega}),$ $\mathbb{P}$ is the two-sided Wiener measure on $(\Omega,\mathscr{F})$. From \cite{FS}, it is clear that  the measure $\mathbb{P}$ is ergodic and invariant under the translation-operator group $\{\theta_t\}_{t\in\R}$ on $\Omega$ defined by
\begin{align*}
	\theta_t \omega(\cdot) = \omega(\cdot+t)-\omega(t), \ \text{ for all }\ t\in\R, \ \omega\in \Omega.
\end{align*}
The operator $\theta(\cdot)$ is known as \emph{Wiener shift operator}. Moreover, the quadruple $(\Omega,\mathscr{F},\mathbb{P},\theta)$ defines a metric dynamical system, see \cite{Arnold,BCLLLR}.

\subsubsection{Ornstein-Uhlenbeck process}
Consider for some $\sigma>0$ (which will be specified later)
\begin{align}\label{OU1}
	z(\theta_{t}\omega) =  \int_{-\infty}^{t} e^{-\sigma(t-\xi)}\d \W(\xi), \ \ \omega\in \Omega,
\end{align} which is the stationary solution of the one dimensional Ornstein-Uhlenbeck equation
\begin{align}\label{OU2}
	\d z(\theta_t\omega) + \sigma z(\theta_t\omega)\d t =\d\W(t).
\end{align}
It is known from \cite{FAN} that there exists a $\theta$-invariant subset $\widetilde{\Omega}\subset\Omega$ of full measure such that $z(\theta_t\omega)$ is continuous in $t$ for every $\omega\in \widetilde{\Omega},$ and
\begin{align}
	\mathbb{E}\left(|z(\theta_s\omega)|^{\xi}\right)&=\frac{\Gamma\left(\frac{1+\xi}{2}\right)}{\sqrt{\pi\sigma^{\xi}}}, \ \text{ for all } \ \xi>0, \ s\in \R,\label{Z2}\\
	\lim_{t\to +\infty} e^{-\delta t}|z(\theta_{-t}\omega)| &=0, \ \text{ for all } \ \delta>0,\label{Z5}\\
	\lim_{t\to \pm \infty} \frac{1}{t} \int_{0}^{t} z(\theta_{\xi}\omega)\d\xi &=\lim_{t\to \pm \infty} \frac{|z(\theta_t\omega)|}{|t|}=0,\label{Z3}
\end{align} where $\Gamma$ is the Gamma function. For further analysis of this work, we do not distinguish between $\widetilde{\Omega}$ and $\Omega$. Since, $\omega(\cdot)$ has sub-exponential growth  (cf. Lemma 11, \cite{CGSV}), $\Omega$ can be written as $\Omega=\cup_{N\in\N}\Omega_{N}$, where
\begin{align*}
	\Omega_{N}:=\{\omega\in\Omega:|\omega(t)|\leq Ne^{|t|},\text{ for all }t\in\R\}, \text{ for all } \ N\in\N.
\end{align*}
\begin{lemma}[Lemma 2.5, \cite{YR}]\label{conv_z}
	For each $N\in\N$, suppose $\omega_k,\omega_0\in\Omega_{N}$ such that $d_{\Omega}(\omega_k,\omega_0)\to0$ as $k\to+\infty$. Then, for each $\tau\in\R$ and $T\in\R^+$ ,
	\begin{align}
		\sup_{t\in[\tau,\tau+T]}&\bigg[|z(\theta_{t}\omega_k)-z(\theta_{t}\omega_0)|+|e^{ z(\theta_{t}\omega_k)}-e^{ z(\theta_{t}\omega_0)}|\bigg]\to 0 \ \text{ as } \ k\to+\infty,\nonumber\\
		\sup_{k\in\N}\sup_{t\in[\tau,\tau+T]}&|z(\theta_{t}\omega_k)|\leq C(\tau,T,\omega_0).\label{conv_z2}
	\end{align}
\end{lemma}

\subsubsection{Backward-uniformly tempered random set}{(\cite{YR})}
A bi-parametric set $\mathcal{D}=\{\mathcal{D}(\tau,\omega)\}$ in a Banach space $\X$ is said to be \emph{backward-uniformly tempered} if
\begin{align}\label{BackTem}
	\lim_{t\to +\infty}e^{-ct}\sup_{s\leq \tau}\|\mathcal{D}(s-t,\theta_{-t}\omega)\|^2_{\X}=0\ \ \forall \ \  (\tau,\omega,c)\in\R\times\Omega\times\R^+,
\end{align}
where $\|\mathcal{D}\|_{\X}=\sup\limits_{\x\in \mathcal{D}}\|\x\|_{\X}.$
\subsubsection{Class of random sets}
\begin{itemize}
	\item Let ${\mathfrak{D}}$ be the collection of subsets of $\H$ defined as:
	\begin{align}\label{D-NSE}
		{\mathfrak{D}}=\left\{{\mathcal{D}}=\{{\mathcal{D}}(\tau,\omega):(\tau,\omega)\in\R\times\Omega\}:\lim_{t\to +\infty}e^{-ct}\sup_{s\leq \tau}\|{\mathcal{D}}(s-t,\theta_{-t}\omega)\|^2_{\H}=0,\ \forall \ c>0\right\}.
	\end{align}
	\item Let ${\mathfrak{B}}$ be the collection of subsets of $\H$ defined as:
\begin{align*}
	{\mathfrak{B}}=\left\{{\mathcal{B}}=\{{\mathcal{B}}(\tau,\omega):(\tau,\omega)\in\R\times\Omega\}:\lim_{t\to +\infty}e^{-ct}\|{\mathcal{B}}(\tau-t,\theta_{-t}\omega)\|^2_{\H}=0,\ \forall\ c>0\right\}.
\end{align*}
	\item Let ${\mathfrak{D}}_{\infty}$ be the collection of subsets of $\H$ defined as:
	\begin{align*}
		{\mathfrak{D}}_{\infty}=\left\{\widehat{\mathcal{D}}=\{\widehat{\mathcal{D}}(\omega):\omega\in\Omega\}:\lim_{t\to +\infty}e^{-\frac{\nu\lambda_{1}}{3}t}\|\widehat{\mathcal{D}}(\theta_{-t}\omega)\|^2_{\H}=0,\ \forall \ c>0\right\}.
	\end{align*}
\end{itemize}
\subsection{Kuratowski's measure of noncompactness} The first result of  measure of noncompactness was defined and studied by Kuratowski in \cite{Kuratowski} (see \cite{Malkowsky} also). With the help of some vital implications of Kuratowski's measure of noncompactness, one can show the existence of a convergent subsequences for some arbitrary sequences. Therefore, several authors used such results to obtain the asymptotic compactness of random dynamical systems, cf. \cite{CGTW,YR} etc. and references therein.
\begin{definition}[Kuratowski's measure of noncompactness, \cite{Rakocevic}]
	Let $(\mathbb{X}, d)$ be a metric space and $E$ a bounded subset of $\mathbb{X}$. Then the Kuratowski measure of noncompactness (the set-measure of noncompactness) of $E$, denoted by $\kappa_{\mathbb{X}}(E)$, is the infimum of the set of all numbers $\varepsilon>0$ such that $E$ can be covered by a finite number of sets with diameters less than $\varepsilon$, that is,
	\begin{align*}
		\kappa_{\mathbb{X}}(E)=\inf\bigg\{\varepsilon>0:E\subset\bigcup_{i=1}^{n}Q_{i}, \ Q_{i}\subset\mathbb{X},\ \mathrm{diam}(Q_{i})<\varepsilon\ \ (i=1,2\ldots,n; \ n\in\N)\bigg\}.
	\end{align*}
The function $\kappa$ is called \emph{Kuratowski's measure of noncompactness}.
\end{definition}
Note that $	\kappa_{\mathbb{X}}(E)=0$ if and only if $\overline{E}$ is compact (see Lemma 1.2, \cite{Rakocevic}). The following lemma is an application of Kuratowski's measure of noncompactness which is helpful in proving the time-semi-uniform asymptotic compactness of random dynamical systems.
\begin{lemma}[Lemma 2.7, \cite{LGL}]\label{K-BAC}
	Let $\mathbb{X}$ be a Banach space and $x_n$ be an arbitrary sequence in $\mathbb{X}$. Then $\{x_n\}$ has a convergent subsequence if $\kappa_{\mathbb{X}}\{x_n:n\geq m\}\to 0 \text{ as } m\to \infty.$
\end{lemma}

	\section{Asymptotically autonomous robustness of random attractors of \eqref{1}: additive noise }\label{sec3}\setcounter{equation}{0}
		In this section, we consider the 2D SNSE \eqref{SNSE} driven by additive white noise, that is, $S(\u)$ is independent of $\u,$ and establish the existence and asymptotic autonomy of $\mathfrak{D}$-pullback random attractors. Let us consider 2D SNSE perturbed by additive white noise for $t\geq \tau,$ $\tau\in\mathbb{R}$ and $\h\in\D(\A)$ as
	\begin{equation}\label{SNSE-A}
		\left\{
		\begin{aligned}
			\frac{\d\u(t)}{\d t}+\nu \A\u(t)+\B(\u(t))&=\f(t) +\h(x)\frac{\d \W(t)}{\d t} , \\
			\u(x,\tau)&=\u_{0}(x),	\ \ \ x\in \mathcal{O},
		\end{aligned}
		\right.
	\end{equation}
	where $\W(t,\omega)$ is the standard scalar Wiener process on the probability space $(\Omega, \mathscr{F}, \mathbb{P})$ (see Section \ref{2.5} above).
	
	Let us define $\v(t,\tau,\omega,\v_{\tau})=\u(t,\tau,\omega,\u_{\tau})-\h(x)z(\theta_{t}\omega)$, where $z$ is defined by \eqref{OU1} and satisfies \eqref{OU2}, and $\u$ is the solution of \eqref{1} with $S(\u)=\h(x)$. Then $\v$ satisfies:
\begin{equation}\label{2-A}
	\left\{
	\begin{aligned}
		\frac{\d\v}{\d t}-\nu \Delta\v&+\big((\v+\h z)\cdot\nabla\big)(\v+\h z)+\nabla p\\&=\f +\sigma\h z+\nu z\Delta\h,   \text{ in }\  \mathcal{O}\times(\tau,\infty), \\ \nabla\cdot\v&=0, \ \ \ \ \ \ \ \ \  \text{ in } \ \ \mathcal{O}\times(\tau,\infty),\\ \v&=0, \ \ \ \ \ \ \ \ \  \text{ in } \ \ \partial\mathcal{O}\times(\tau,\infty), \\
		\v(x,\tau)&=\v_{0}(x)=\u_{0}(x)-\h(x)z(\theta_{\tau}\omega),   \ x\in \mathcal{O} \ \text{ and }\ \tau\in\R,
	\end{aligned}
	\right.
\end{equation}
as well as the projected form in $\V^*$:
\begin{equation}\label{CNSE-A}
	\left\{
	\begin{aligned}
		\frac{\d\v}{\d t} +\nu \A\v+ \B(\v+\h z)&= \boldsymbol{f} + \sigma\h z -\nu z\A\h , \quad t> \tau,   \ \tau\in\R ,\\
		\v(x,\tau)&=\v_{0}(x)=\u_{0}(x)-\h(x)z(\theta_{\tau}\omega), \ \ x\in\mathcal{O},
	\end{aligned}
	\right.
\end{equation}

\subsection{Lusin continuity and measurability of systems}
Lusin continuity assists us to define the non-autonomous random dynamical system (NRDS). The following Lemma gives us the energy inequality satisfied by the solution of the system \eqref{CNSE-A} which will be frequently used.
\begin{lemma}
	Assume that $\f\in\mathrm{L}^2_{\mathrm{loc}}(\R;\H)$, Hypotheses \ref{assumpO} and \ref{AonH-N} are satisfied. Then, the solution of \eqref{CNSE-A} satisfies the following inequality:
	\begin{align}\label{EI1-N}
		&\frac{\d}{\d t}\|\v(t)\|^2_{\H}+\left(\nu\lambda-4{\aleph}\left|z(\theta_{t}\omega)\right|\right)\|\v(t)\|^2_{\H}+\frac{\nu}{2}\|\v(t)\|^2_{\V}\leq \widehat{R}_4\left[\|\f(t)\|^{2}_{\H}+\left|z(\theta_{t}\omega)\right|^3+1\right],
	\end{align}
	where $\widehat{R}_4>0$ is some constant.
\end{lemma}
\begin{proof}
	From \eqref{CNSE-A}, we obtain
	\begin{align}\label{ue17-N}
		\frac{1}{2}\frac{\d}{\d t}\|\v\|^2_{\H}=&-\nu\|\v\|^2_{\V}-b(\v+\h z(\theta_{t}\omega),\v+\h z(\theta_{t}\omega),\v)+\left(\f,\v\right)+z(\theta_{t}\omega)\left(\sigma\h-\nu \A\h,\v\right).
	\end{align}
	Making use of \eqref{b0} and Hypothesis \ref{AonH-N}, and  we find the existence of a constant $\widehat{R}_1>0$ such that
	\begin{align}\label{ue18-N}
		\left|b(\v+\h z(\theta_{t}\omega),\v+\h z(\theta_{t}\omega),\v)\right|\leq2{\aleph}\left|z(\theta_{t}\omega)\right|\|\v\|^2_{\H}+\widehat{R}_1\left|z(\theta_{t}\omega)\right|^3.
	\end{align}
	Using \eqref{poin}, H\"older's and Young's inequalities, there exist constants $\widehat{R}_2,\widehat{R}_3>0$ such that
	\begin{align}
		\left(\f,\v\right)+z(\theta_{t}\omega)\big(\sigma\h-\nu \A\h,\v\big)\leq\frac{\nu\lambda}{2}\|\v\|^2_{\H}+\widehat{R}_2\|\f\|^{2}_{\H}+\widehat{R}_3\left[\left|z(\theta_{t}\omega)\right|^3+1\right].\label{ue19-N}
	\end{align}
	Combining \eqref{ue17-N}-\eqref{ue19-N} and using \eqref{poin}, we complete the proof with $\widehat{R}_4=2\max\{\widehat{R}_1+\widehat{R}_3,\widehat{R}_2,\widehat{R}_3\}$.
\end{proof}

\begin{lemma}\label{Soln-N}
	Suppose that $\f\in\mathrm{L}^2_{\mathrm{loc}}(\R;\H)$. For each $(\tau,\omega,\v_{\tau})\in\R\times\Omega\times\H$, the system \eqref{CNSE-A} has a unique solution $\v(\cdot,\tau,\omega,\v_{\tau})\in\mathrm{C}([\tau,+\infty);\H)\cap\mathrm{L}^2_{\mathrm{loc}}(\tau,+\infty;\V)$ such that $\v$ is continuous with respect to the  initial data.
\end{lemma}
\begin{proof}
The proof can be found in Theorems 4.5 and 4.6 in \cite{BL}.
\end{proof}

The next result shows the Lusin continuity of mapping of solution to the system \eqref{CNSE-A} in sample points.
\begin{proposition}\label{LusinC-N}
	Suppose that $\f\in\mathrm{L}^2_{\mathrm{loc}}(\R;\H)$ and Hypotheses \ref{assumpO} and \ref{AonH-N} are satisfied. For each $N\in\N$, the mapping $\omega\mapsto\v(t,\tau,\omega,\v_{\tau})$ $($solution of \eqref{CNSE-A}$)$ is continuous from $(\Omega_{N},d_{\Omega_N})$ to $\H$, uniformly in $t\in[\tau,\tau+T]$ with $T>0.$
\end{proposition}
\begin{proof}
	Assume $\omega_k,\omega_0\in\Omega_N$ such that $d_{\Omega_N}(\omega_k,\omega_0)\to0$ as $k\to\infty$. Let $\mathscr{V}^k:=\v^k-\v^0,$ where $\v^k=\v(t,\tau,\omega_k,\v_{\tau})$ and $\v_0=\v(t,\tau,\omega_0,\v_{\tau})$ for $t\in[\tau,\tau+T]$. Then, $\mathscr{V}^k$ satisfies:
	\begin{align}\label{LC1-N}
		\frac{\d\mathscr{V}^k}{\d t}&+\nu \A\mathscr{V}^k+\left[\B\big(\v^k+z(\theta_{t}\omega_k)\h\big)-\B\big(\v^0+z(\theta_{t}\omega_0)\h\big)\right]\nonumber\\&=\left\{\sigma\h-\nu\A\h\right\}\left[z(\theta_t\omega_k)-z(\theta_t\omega_0)\right],
	\end{align}
	in $\V^*$.  Taking the  inner product with $\mathscr{V}^k(\cdot)$ in \eqref{LC1-N} and using \eqref{441}, we get
	\begin{align}\label{LC2-N}
		\frac{1}{2}\frac{\d }{\d t}\|\mathscr{V}^k\|^2_{\H}&=\nu\|\mathscr{V}^k\|^2_{\V}-\left[z(\theta_{t}\omega_k)-z(\theta_{t}\omega_0)\right]b\big(\v^k+z(\theta_t\omega_k)\h,\h,\v^k+z(\theta_t\omega_k)\h\big)\nonumber\\&\quad-b\big(\mathscr{V}^k+\left[z(\theta_t\omega_k)-z(\theta_t\omega_0)\right]\h,\mathscr{V}^k+\left[z(\theta_t\omega_k)-z(\theta_t\omega_0)\right]\h,\v^0+z(\theta_t\omega_0)\h\big)\nonumber\\&\quad+\left[z(\theta_{t}\omega_k)-z(\theta_{t}\omega_0)\right]b\big(\v^0+z(\theta_t\omega_0)\h,\h,\v^0+z(\theta_t\omega_0)\h\big)\nonumber\\&\quad+\left[z(\theta_t\omega_k)-z(\theta_t\omega_0)\right]\big(\sigma\h-\nu\A\h,\mathscr{V}^k\big).
	\end{align}
	In view of Hypothesis \ref{AonH-N}, \eqref{poin}, H\"older's and Young's inequalities, we obtain
	\begin{align}
		&\bigg|\left[z(\theta_{t}\omega_k)-z(\theta_{t}\omega_0)\right]\bigg\{b\big(\v^k+z(\theta_t\omega_k)\h,\h,\v^k+z(\theta_t\omega_k)\h\big)\nonumber\\&+b\big(\v^0+z(\theta_t\omega_0)\h,\h,\v^0+z(\theta_t\omega_0)\h\big)+\big(\sigma\h-\nu\A\h,\mathscr{V}^k\big)\bigg\}\bigg|\nonumber\\& \leq C\left|z(\theta_t\omega_k)-z(\theta_t\omega_0)\right|\bigg\{\|\v^k+z(\theta_t\omega_k)\h\|^2_{\V}+\|\v^0+z(\theta_t\omega_0)\h\|^2_{\V}+\|\v^k\|^2_{\V}+\|\v^0\|^2_{\V}+1\bigg\}\label{LC6-N}.
	\end{align}
	Next, we estimate the remaining term on the right hand side of \eqref{LC2-N}. Applying \eqref{poin}, \eqref{b0}, \eqref{b1} and Young's inequality, we estimate
	\begin{align}
		&\left|b\big(\mathscr{V}^k+\left[z(\theta_t\omega_k)-z(\theta_t\omega_0)\right]\h,\mathscr{V}^k+\left[z(\theta_t\omega_k)-z(\theta_t\omega_0)\right]\h,\v^0+z(\theta_t\omega_0)\h\big)\right|\nonumber\\&=\left|b\big(\mathscr{V}^k+\left[z(\theta_t\omega_k)-z(\theta_t\omega_0)\right]\h,\v^0+z(\theta_t\omega_0)\h,\mathscr{V}^k+\left[z(\theta_t\omega_k)-z(\theta_t\omega_0)\right]\h\big)\right|\nonumber\\&\leq C\|\v^0+z(\theta_t\omega_0)\h\|^2_{\V}\|\mathscr{V}^k\|^2_{\H}+C\left|z(\theta_t\omega_k)-z(\theta_t\omega_0)\right|^2\bigg\{\|\v^0\|^2_{\V}+\left|z(\theta_t\omega_0)\right|^2+1\bigg\}+\frac{\nu}{2}\|\mathscr{V}^k\|^2_{\V}.\label{LC7-N}
	\end{align}
	Combining \eqref{LC2-N}-\eqref{LC7-N}, we obtain
	\begin{align}\label{LC8-N}
		\frac{\d }{\d t}\|\mathscr{V}^k(t)\|^2_{\H}\leq P(t)\|\mathscr{V}^k(t)\|^2_{\H}+Q_k(t),
	\end{align}
	for a.e. $t\in[\tau,\tau+T]$, $T>0$, where $P=C\|\v^0+z(\theta_t\omega_0)\h\|^2_{\V}$ and
	\begin{align*}
		Q_k&=C\left|z(\theta_{t}\omega_k)-z(\theta_{t}\omega_0)\right|\bigg\{\|\v^k+z(\theta_t\omega_k)\h\|^2_{\V}+\|\v^0+z(\theta_t\omega_0)\h\|^2_{\V}+\|\v^k\|^2_{\V}+\|\v^0\|^2_{\V}+1\bigg\}\nonumber\\&\qquad+C\left|z(\theta_t\omega_k)-z(\theta_t\omega_0)\right|^2\bigg\{\|\v^0\|^2_{\V}+\left|z(\theta_t\omega_0)\right|^2+1\bigg\}.
	\end{align*}
We infer from \eqref{EI1-N} that for all $t\in[\tau,\tau+T]$,
\begin{align}\label{LC9-N}
	\frac{\d}{\d t}\|\v^k(t)\|^2_{\H}+\frac{\nu}{2}\|\v^k(t)\|^2_{\V}&\leq4{\aleph}\left|z(\theta_{t}\omega_k)\right|\|\v^k(t)\|^2_{\H}+ \widehat{R}_4\left[\|\f(t)\|^{2}_{\H}+\left|z(\theta_{t}\omega_k)\right|^3+1\right]\nonumber\\&\leq C(\tau,T,\omega_0)\|\v^k(t)\|^2_{\H}+ C(\tau,T,\omega_0)\left[\|\f(t)\|^{2}_{\H}+1\right],
\end{align}
where we have used \eqref{conv_z2}. Applying Gronwall's inequality, we arrive at
\begin{align}\label{LC10-N}
&\sup_{k\in\N}\sup_{t\in[\tau,\tau+T]}\|\v^k(t)\|^2_{\H}\leq e^{CT}\left[\|\v_{\tau}\|^2_{\H}+ C(\tau,T,\omega_0)\int_{\tau}^{\tau+T}\left(\|\f(\xi)\|^{2}_{\H}+1\right)\d t\right]\leq C(\tau,T,\omega_0),
\end{align}
where we have used the fact $\f\in\mathrm{L}^2_{\text{loc}}(\R;\H)$. Further, integrating \eqref{LC9-N} from $\tau$ to $\tau+T$ and using \eqref{LC10-N}, we get
\begin{align}\label{LC11-N}
	\sup_{k\in\N}\int_{\tau}^{\tau+T}\|\v^k(t)\|^2_{\V}\d t\leq C(\tau,T,\omega_0).
\end{align}
Now, from \eqref{conv_z2}, \eqref{LC11-N}, $\f\in\mathrm{L}^2_{\text{loc}}(\R;\H)$, $\v^0\in\mathrm{C}([\tau,+\infty);\H)\cap\mathrm{L}^2_{\mathrm{loc}}(\tau,+\infty;\V)$ and Lemma \ref{conv_z}, we conclude that
\begin{align}\label{LC13-N}
\int_{\tau}^{\tau+T}P(t)\d t\leq C(\tau,T,\omega_0)\ \  \text{ and }\ \ 	\lim_{k\to+\infty}\int_{\tau}^{\tau+T}Q_k(t)\d t=0.
\end{align}
Making use of Gronwall's inequality to \eqref{LC8-N} and using \eqref{LC13-N}, one can complete the proof.
\end{proof}

Note that Lemma \ref{Soln-N} ensures that we can define a mapping $\Phi:\R^+\times\R\times\Omega\times\H\to\H$ by
\begin{align}\label{Phi-N}
	\Phi(t,\tau,\omega,\u_{\tau})=\u(t+\tau,\tau,\theta_{-\tau}\omega,\u_{\tau})=\v(t+\tau,\tau,\theta_{-\tau}\omega,\v_{\tau})+\h z(\theta_{t}\omega).
\end{align}
The Lusin continuity in Proposition \ref{LusinC-N} provides the $\mathscr{F}$-measurability of $\Phi$. Consequently, Lemma \ref{Soln-N} and Proposition \ref{LusinC-N} imply that the mapping $\Phi$ defined by \eqref{Phi-N} is an NRDS on $\H$.

\subsection{Backward convergence of NRDS}
Consider the following autonomous 2D SNSE subjected to an additive white noise:
\begin{equation}\label{A-SNSE}
	\left\{
	\begin{aligned}
		\frac{\d\widetilde{\u}(t)}{\d t}+\nu \A\widetilde{\u}(t)+\B(\widetilde{\u}(t))&=\f_{\infty} +\h(x)\frac{\d \W(t)}{\d t}, \\
		\widetilde{\u}(x,0)&=\widetilde{\u}_{0}(x),\ \	x\in \mathcal{O}.
	\end{aligned}
	\right.
\end{equation}

We show that the solution to the system \eqref{CNSE-A} converges to the solution of the corresponding autonomous system \eqref{A-SNSE} as $\tau\to-\infty$. Let $\widetilde{\v}(t,\omega)=\widetilde{\u}(t,\omega)-\h(x)z(\theta_{t}\omega)$. Then, the pathwise deterministic system corresponding to the stochastic system \eqref{A-SNSE} is given by:
\begin{eqnarray}\label{A-CNSE}
	\left\{
	\begin{aligned}
		\frac{\d\widetilde{\v}(t)}{\d t} +\nu \A\widetilde{\v}(t)+ \B(\widetilde{\v}(t)+\h z(\theta_{t}\omega)) &= {\boldsymbol{f}}_{\infty} + \sigma\h z(\theta_{t}\omega) -\nu z(\theta_{t}\omega)\A\h ,  t> \tau, \tau\in\R ,\\
		\widetilde{\v}(x,0)&=\widetilde{\v}_{0}(x)=\widetilde{\u}_{0}(x)-\h(x)z(\omega), \ \ x\in\mathcal{O},
	\end{aligned}
	\right.
\end{eqnarray}
in $\V^*$.

The proof of the following result is just similar to that of Proposition 4.3 in \cite{KRM}.
\begin{proposition}\label{Back_conver-N}
	Suppose that Hypotheses \ref{assumpO} and \ref{Hypo_f-N} are satisfied. Then the solution $\v$ of the system \eqref{CNSE-A} backward converges to the solution $\widetilde{\v}$ of the system \eqref{A-CNSE}, that is,
	\begin{align*}
		\lim_{\tau\to -\infty}\|\v(T+\tau,\tau,\theta_{-\tau}\omega,\v_{\tau})-\widetilde{\v}(t,\omega,\widetilde{\v}_0)\|_{\H}=0, \ \ \text{ for all } T>0 \text{ and } \omega\in\Omega,
	\end{align*}
	whenever $\|\v_{\tau}-\widetilde{\v}_0\|_{\H}\to0$ as $\tau\to-\infty.$
\end{proposition}

\subsection{Increasing random absorbing sets}
This subsection provides the existence of an increasing ${\mathfrak{D}}$-random absorbing set for the non-autonomous SNSE.

\begin{lemma}\label{Absorbing-N}
	Suppose that $\f\in\mathrm{L}^2_{\mathrm{loc}}(\R;\H)$, Hypotheses \ref{assumpO} and \ref{AonH-N} are satisfied. Then, for all $(\tau,\omega)\in\R\times\Omega,$ $s\leq\tau$, $\xi\geq s-t,$ $ t\geq0$ and $\v_{0}\in\H$,
	\begin{align}\label{AB2-N}
		&\|\v(\xi,s-t,\theta_{-s}\omega,\v_{0})\|^2_{\H}+\frac{\nu}{2}\int_{s-t}^{\xi}e^{\nu\lambda(\uprho-\xi)-4{\aleph}\int^{\uprho}_{\xi}\left|z(\theta_{\upeta-s}\omega)\right|\d\upeta}\|\v(\uprho,s-t,\theta_{-s}\omega,\v_{0})\|^2_{\V}\d\uprho\nonumber\\&\leq e^{-\nu\lambda(\xi-s+t)+4{\aleph}\int_{-t}^{\xi-s}\left|z(\theta_{\upeta}\omega)\right|\d\upeta}\|\v_{0}\|^2_{\H}\nonumber\\&\quad+\widehat{R}_4\int_{-t}^{\xi-s}e^{\nu\lambda(\uprho+s-\xi)-4{\aleph}\int^{\uprho}_{\xi-s}\left|z(\theta_{\upeta}\omega)\right|\d\upeta}\bigg\{\|\f(\uprho+s)\|^2_{\H}+\left|z(\theta_{\uprho}\omega)\right|^3+1\bigg\}\d\uprho,
	\end{align}
	where $\widehat{R}_4$ is the same as in \eqref{EI1-N}. For each $(\tau,\omega,D)\in\R\times\Omega\times{\mathfrak{D}},$ there exists a time $\mathfrak{T}:=\mathfrak{T}(\tau,\omega,D)>0$ such that
	\begin{align}\label{AB1-N}
		&\sup_{s\leq\tau}\sup_{t\geq \mathfrak{T}}\sup_{\v_{0}\in D(s-t,\theta_{-t}\omega)}\bigg[\|\v(s,s-t,\theta_{-s}\omega,\v_{0})\|^2_{\H}\nonumber\\&\quad+\frac{\nu}{2}\int_{s-t}^{s}e^{\nu\lambda(\uprho-s)-4{\aleph}\int^{\uprho}_{s}\left|z(\theta_{\upeta-s}\omega)\right|\d\upeta}\|\v(\uprho,s-t,\theta_{-s}\omega,\v_{0})\|^2_{\V}\d\uprho\bigg]\leq 1+\widehat{R}_4 \sup_{s\leq \tau}R(s,\omega),
	\end{align}
	where $R(s,\omega)$ is given by
	\begin{align}\label{ABr-N}
		R(s,\omega):=\int_{-\infty}^{0}e^{\nu\lambda\uprho-4{\aleph}\int^{\uprho}_{0}\left|z(\theta_{\upeta}\omega)\right|\d\upeta}\bigg\{\|\f(\uprho+s)\|^2_{\H}+\left|z(\theta_{\uprho}\omega)\right|^3+1\bigg\}\d\uprho.
	\end{align}
\end{lemma}
\begin{proof}
		Let us write the energy inequality \eqref{EI1-N} for $\v(\zeta)=\v(\zeta,s-t,\theta_{-s}\omega,\v_{0})$, that is,
	\begin{align}\label{AB-EI}
		&\frac{\d}{\d \zeta}\|\v(\zeta)\|^2_{\H}+\left(\nu\lambda-4\aleph\left|z(\theta_{\zeta-s}\omega)\right|\right)\|\v(\zeta)\|^2_{\H}+\frac{\nu}{2}\|\v(\zeta)\|^2_{\V}\leq \widehat{R}_4\left[\|\f(\zeta)\|^{2}_{\H}+\left|z(\theta_{\zeta-s}\omega)\right|^3+1\right].
	\end{align}
	In view of the variation of constants formula with respect to $\zeta\in(s-t,\xi)$, we get \eqref{AB2-N} immediately. Putting $\xi=s$ in \eqref{AB2-N}, we obtain
	\begin{align}\label{AB4-N}
		&\|\v(s,s-t,\theta_{-s}\omega,\v_{0})\|^2_{\H}+\frac{\nu}{2}\int_{s-t}^{s}e^{\nu\lambda(\uprho-s)-4{\aleph}\int^{\uprho}_{s}\left|z(\theta_{\upeta-s}\omega)\right|\d\upeta}\|\v(\uprho,s-t,\theta_{-s}\omega,\v_{0})\|^2_{\V}\d\uprho\nonumber\\&\leq e^{-\nu\lambda t +4{\aleph}\int_{-t}^{0}\left|z(\theta_{\upeta}\omega)\right|\d\upeta}\|\v_{0}\|^2_{\H}\nonumber\\&\quad+\widehat{R}_4\int_{-t}^{0}e^{\nu\lambda\uprho-4{\aleph}\int^{\uprho}_{0}\left|z(\theta_{\upeta}\omega)\right|\d\upeta}\bigg\{\|\f(\uprho+s)\|^2_{\H}+\left|z(\theta_{\uprho}\omega)\right|^3+1\bigg\}\d\uprho,
	\end{align}
		for all $s\leq\tau$. Now, we consider $\sigma$ large enough $\left(\sigma>\frac{9216\aleph^2}{\pi\nu^2\lambda^2}\right)$ such that from \eqref{Z2}, we have
	\begin{align}\label{Z6}
		4\aleph\mathbb{E}(|\z(\cdot)|)<\frac{\nu\lambda}{24}<\frac{2\nu\lambda}{3}.
	\end{align}
	Since $\v_0\in D(s-t,\theta_{-t}\omega)$ and $D$ is backward tempered, it implies from \eqref{Z6} and the definition of backward temperedness \eqref{BackTem} that there exists a time ${\mathfrak{T}}:={\mathfrak{T}}(\tau,\omega,D)$ such that for all $t\geq {\mathfrak{T}}>0$,
	\begin{align}\label{v_0-1}
		e^{-\nu\lambda t+4\aleph\int_{-t}^{0}\left|z(\theta_{\upeta}\omega)\right|\d\upeta}\sup_{s\leq \tau}\|\v_{0}\|^2_{\H}\leq e^{-\frac{\nu\lambda}{3}t}\sup_{s\leq \tau}\|D(s-t,\theta_{-t}\omega)\|^2_{\H}\leq1.
	\end{align}
	Taking supremum over $s\in(-\infty,\tau]$ in \eqref{AB4-N}, one obtains \eqref{AB1-N}.
\end{proof}

\begin{proposition}\label{IRAS-N}
	Suppose that $\f\in\mathrm{L}^2_{\mathrm{loc}}(\R;\H)$, Hypotheses \ref{assumpO} and \ref{AonH-N} are satisfied. For $\widehat{R}_4$ and $R(\tau,\omega)$ same as in \eqref{EI1-N} and \eqref{ABr-N}, respectively, we have
	\vskip 2mm
	\noindent
\emph{(i)} There is an increasing $\mathfrak{D}$-pullback absorbing set $\mathcal{R}$ given by
	\begin{align}\label{IRAS1-N}
		\mathcal{R}(\tau,\omega):=\left\{\u\in\H:\|\u\|^2_{\H}\leq 2+2\widehat{R}_4\sup_{s\leq \tau} R(s,\omega)+2\|\h\|^2_{\H}\left|z(\omega)\right|^2\right\}, \text{ for all } \tau\in\R.
	\end{align}
 Moreover, $\mathcal{R}$ is backward-uniformly  tempered with arbitrary rate, that is, $\mathcal{R}\in{\mathfrak{D}}$.
	\vskip 2mm
	\noindent
	\emph{(ii)} There is a $\mathfrak{B}$-pullback \textbf{random} absorbing set $\widetilde{\mathcal{R}}$ given by
	\begin{align}\label{IRAS11-N}
		\widetilde{\mathcal{R}}(\tau,\omega):=\left\{\u\in\H:\|\u\|^2_{\H}\leq 2+2\widehat{R}_4 R(s,\omega)+2\|\h\|^2_{\H}\left|z(\omega)\right|^2\right\}\in{\mathfrak{B}}, \text{ for all } \tau\in\R.
	\end{align}
\end{proposition}
\begin{proof}
		(i) Using \eqref{G3}, \eqref{Z3} and \eqref{Z6}, we obtain
	\begin{align}\label{IRAS2-N}
		\sup_{s\leq \tau}R(s,\omega)&= \sup_{s\leq \tau}\int_{-\infty}^{0}
e^{\nu\lambda\uprho-4\aleph\int^{\uprho}_{0}
\left|z(\theta_{\upeta}\omega)\right|\d\upeta}
\bigg\{\|\f(\uprho+s)\|^2_{\H}
+\left|z(\theta_{\uprho}\omega)
\right|^3+1\bigg\}\d\uprho\nonumber\\&
=\sup_{s\leq \tau}\int_{-\infty}^{0}e^{\frac{\nu\lambda}{3}\uprho}\bigg\{\|\f(\uprho+s)\|^2_{\H}+\left|z(\theta_{\uprho}\omega)\right|^3+1\bigg\}\d\uprho<\infty.
	\end{align}
Hence, absorption follows from Lemma \ref{Absorbing-N}. Due to the fact that $\tau\mapsto \sup_{s\leq\tau}R(\tau,\omega)$ is an increasing function, $\mathcal{R}(\tau,\omega)$ is an increasing $\mathfrak{D}$-pullback absorbing set. For $c>0$, let $c_1=\min\{\frac{c}{2},\frac{\nu\lambda}{3}\}$ and consider
\begin{align}\label{IRAS3-N}
	&\lim_{t\to+\infty}e^{-ct}\sup_{s\leq \tau}\|\mathcal{R}(s-t,\theta_{-t}\omega)\|^2_{\H}\nonumber\\&\leq \lim_{t\to+\infty}e^{-ct}\left[2+2\widehat{R}_4\sup_{s\leq \tau} R(s-t,\theta_{-t}\omega)+2\|\h\|^2_{\H}\left|z(\theta_{-t}\omega)\right|^2\right] \nonumber\\&= 2\widehat{R}_4\lim_{t\to+\infty}e^{-ct}\sup_{s\leq \tau}\int_{-\infty}^{0}e^{\frac{\nu\lambda}{3}\uprho}\bigg\{\|\f(\uprho+s-t)\|^2_{\H}+\left|z(\theta_{\uprho-t}\omega)\right|^3+1\bigg\}\d\uprho \nonumber\\&= 2\widehat{R}_4\lim_{t\to+\infty}e^{-ct}\sup_{s\leq \tau}\int_{-\infty}^{-t}e^{\frac{\nu\lambda}{3}(\uprho+t)}\bigg\{\|\f(\uprho+s)\|^2_{\H}+\left|z(\theta_{\uprho}\omega)\right|^3+1\bigg\}\d\uprho \nonumber\\&\leq 2\widehat{R}_4\lim_{t\to+\infty}e^{-(c-c_1)t}\sup_{s\leq \tau}\int_{-\infty}^{0}e^{c_1\uprho}\bigg\{\|\f(\uprho+s)\|^2_{\H}+\left|z(\theta_{\uprho}\omega)\right|^3+1\bigg\}\d\uprho=0,
\end{align}
where we have used  \eqref{G3}, \eqref{Z3} and \eqref{IRAS2-N}. It infers from \eqref{IRAS3-N} that $\mathcal{R}\in{\mathfrak{D}}$.
\vskip 2mm
\noindent
(ii) Since $\widetilde{\mathcal{R}}\subseteq\mathcal{R}\in\mathfrak{D}\subseteq\mathfrak{B}$ and the mapping $\omega\mapsto R(\tau,\omega)$ is $\mathscr{F}$-measurable, using \eqref{AB4-N} (for $s=\tau$), we obtain that $\widetilde{\mathcal{R}}$ is a $\mathfrak{B}$-pullback \textbf{random} absorbing set.
\end{proof}
\subsection{Backward uniform-tail estimates and backward flattening estimates}
Backward uniform tail estimates and backward flattening estimates for the solution of the system \eqref{CNSE-A} play a key role in establishing the time-semi-uniform asymptotic compactness (BAC) of NRDS \eqref{Phi-N}. We obtain these estimates by using a proper cut-off function.
\begin{lemma}\label{largeradius-N}
	Suppose that Hypotheses \ref{poin} and \ref{AonH-N} are satisfied. Then, for any $(\tau,\omega,D)\in\R\times\Omega\times{\mathfrak{D}},$ the solution of \eqref{CNSE-A}  satisfies
	\begin{align}\label{ep-N}
		&\lim_{k,t\to+\infty}\sup_{s\leq \tau}\sup_{\v_{0}\in D(s-t,\theta_{-t}\omega)}\|\v(s,s-t,\theta_{-s}\omega,\v_{0})\|^2_{\L^2(\mathcal{O}^{c}_{k})}=0,
	\end{align}
	where $\mathcal{O}^c_{k}=\mathcal{O}\backslash\mathcal{O}_k$ and $\mathcal{O}_{k}=\{x\in\mathcal{O}:|x|\leq k\}$.
\end{lemma}
\begin{proof}
		Let $\uprho$ be a smooth function such that $0\leq\uprho(\xi)\leq 1$ for $\xi\in\R^+$ and
	\begin{align}\label{337}
		\uprho(\xi)=\begin{cases*}
			0, \text{ for }0\leq \xi\leq 1,\\
			1, \text{ for } \xi\geq2.
		\end{cases*}
	\end{align}
	Then, there exists a positive constant $C$ such that $|\uprho'(\xi)|\leq C$ and $|\uprho''(\xi)|\leq C$ for all $\xi\in\R^+$. Taking the divergence to the first equation of \eqref{2-A}, formally,  we obtain
	\begin{align*}
		-\Delta p&=\nabla\cdot\left[\big((\v+{\boldsymbol{h}}z(\theta_{t}\omega))\cdot\nabla\big)(\v+{\boldsymbol{h}}z(\theta_{t}\omega))\right]\\&=\nabla\cdot\left[\nabla\cdot\big((\v+\boldsymbol{h}z(\theta_{t}\omega))\otimes(\v+\boldsymbol{h}z(\theta_{t}\omega))\big)\right]\\
		&=\sum_{i,j=1}^{2}\frac{\partial^2}{\partial x_i\partial x_j}\big((v_i+{h}_iz(\theta_{t}\omega))(v_j+{h}_jz(\theta_{t}\omega))\big),
	\end{align*}
which implies that
	\begin{align}\label{p-value}
		p=(-\Delta)^{-1}\left[\sum_{i,j=1}^{2}\frac{\partial^2}{\partial x_i\partial x_j}\big((v_i+h_i z(\theta_{t}\omega))(v_j+h_j z(\theta_{t}\omega))\big)\right],
	\end{align}
in the weak sense. It follows from \eqref{p-value} that
\begin{align}
	\|p\|^2_{\mathrm{L}^2(\mathcal{O})}&=\left\|\left[\sum_{i,j=1}^{2}\frac{\partial^2}{\partial x_i\partial x_j}(-\Delta)^{-1}\big((v_i+h_i z(\theta_{t}\omega))(v_j+h_j z(\theta_{t}\omega))\big)\right]\right\|^2_{\mathrm{L}^2(\mathcal{O})}\nonumber\\  &\leq C\left\|\sum_{i,j=1}^{2}(-\Delta)^{-1}\big((v_i+h_i z(\theta_{t}\omega))(v_j+h_j z(\theta_{t}\omega))\big)
\right\|^2_{\H^{2}(\mathcal{O})}\nonumber\\  &\leq C\left\|\Delta\sum_{i,j=1}^{2}(-\Delta)^{-1}\big((v_i+h_i z(\theta_{t}\omega))(v_j+h_j z(\theta_{t}\omega))\big)
\right\|^2_{\L^{2}(\mathcal{O})}
\nonumber\\  &\leq C \|\v+\h(\theta_{t}\omega)\|^4_{\L^4(\mathcal{O})}\label{p-value-N},
\end{align}
where, in the penultimate step, we have used the elliptic regularity for Poincar\'e domains with uniformly smooth boundary of class $\mathrm{C}^3$ (cf. Lemmas 1, \cite{Heywood}). Taking the inner product to the first equation of \eqref{2-A} with $\uprho^2\left(\frac{|x|^2}{k^2}\right)\v$ in $\mathbb{L}^2(\mathcal{O})$, we have
	\begin{align}\label{ep1-N}
		&\frac{1}{2} \frac{\d}{\d t}\int_{\mathcal{O}}\uprho^2\left(\frac{|x|^2}{k^2}\right)|\v|^2\d x\nonumber \\&= \underbrace{\nu\int_{\mathcal{O}}(\Delta\v) \uprho^2\left(\frac{|x|^2}{k^2}\right) \v \d x}_{:=I_1(k,t)}-\underbrace{b\left(\v+\h z(\theta_{t}\omega),\v+\h z(\theta_{t}\omega),\uprho^2\left(\frac{|x|^2}{k^2}\right)(\v+\h z(\theta_{t}\omega))\right)}_{:=I_2(k,t)}\nonumber\\&\quad+\underbrace{b\left(\v+\h z(\theta_{t}\omega),\v+\h z(\theta_{t}\omega),\uprho^2\left(\frac{|x|^2}{k^2}\right)\h z(\theta_{t}\omega)\right)}_{:=I_3(k,t)}-\underbrace{\int_{\mathcal{O}}(\nabla p)\uprho^2\left(\frac{|x|^2}{k^2}\right)\v\d x}_{:=I_4(k,t)}\nonumber\\&\quad+ \underbrace{\int_{\mathcal{O}}\f\uprho^2\left(\frac{|x|^2}{k^2}\right)\v\d x +\sigma z(\theta_{t}\omega)\int_{\mathcal{O}}
\h\uprho^2\left(\frac{|x|^2}{k^2}\right)\v\d x+\nu z(\theta_{t}\omega)\int_{\mathcal{O}}(\Delta\h)\uprho^2\left(\frac{|x|^2}{k^2}\right)\v\d x}_{:=I_5(k,t)}.
	\end{align}
	Let us now estimate each term on the right hand side of \eqref{ep1-N}. Integration by parts, divergence free condition of $\v(\cdot)$ and \eqref{poin} give
	\begin{align}\label{ep2-N}
		I_{1}(k,t)&= -\nu \int_{\mathcal{O}}\left|\nabla\left(\uprho\left(\frac{|x|^2}{k^2}\right) \v\right)\right|^2  \d x+\nu \int_{\mathcal{O}}\v\nabla\left(\uprho\left(\frac{|x|^2}{k^2}\right)\right)\nabla\left(\uprho\left(\frac{|x|^2}{k^2}\right) \v \right) \d x \nonumber\\&\quad-\nu \int_{\mathcal{O}}\nabla\v \nabla\left(\uprho\left(\frac{|x|^2}{k^2}\right)\right)\uprho\left(\frac{|x|^2}{k^2}\right) \v  \d x\nonumber\\&\leq-\nu \int_{\mathcal{O}}\left|\nabla\left(\uprho\left(\frac{|x|^2}{k^2}\right) \v\right)\right|^2\d x+\frac{\nu}{8} \int_{\mathcal{O}}\left|\nabla\left(\uprho\left(\frac{|x|^2}{k^2}\right) \v\right)\right|^2\d x\nonumber\\&\quad+\frac{\nu\lambda}{8} \int_{\mathcal{O}}\left|\left(\uprho\left(\frac{|x|^2}{k^2}\right) \v\right)\right|^2\d x+\frac{C}{k}\left[\|\v\|^2_{\H}+\|\v\|^2_{\V}\right]\nonumber\\&\leq-\frac{3\nu\lambda}{4} \int_{\mathcal{O}}\left|\left(\uprho\left(\frac{|x|^2}{k^2}\right) \v\right)\right|^2\d x+\frac{C}{k}\|\v\|^2_{\V},
	\end{align}
	and
	\begin{align}\label{ep3-N}
		-I_{2}(k,t)&=4\int_{\mathcal{O}} \uprho\left(\frac{|x|^2}{k^2}\right)\uprho'\left(\frac{|x|^2}{k^2}\right)\frac{x}{k^2}\cdot(\v+\h z(\theta_{t}\omega)) |\v+\h z(\theta_{t}\omega)|^2 \d x\nonumber\\&= 4 \int\limits_{\mathcal{O}\cap\{k\leq|x|\leq \sqrt{2}k\}}\uprho\left(\frac{|x|^2}{k^2}\right) \uprho'\left(\frac{|x|^2}{k^2}\right)\frac{x}{k^2}\cdot(\v+\h z(\theta_{t}\omega)) |\v+\h z(\theta_{t}\omega)|^2 \d x\nonumber\\&\leq \frac{4\sqrt{2}}{k} \int\limits_{\mathcal{O}\cap\{k\leq|x|\leq \sqrt{2}k\}} \left|\uprho'\left(\frac{|x|^2}{k^2}\right)\right| |\v+\h z(\theta_{t}\omega)|^3 \d x\nonumber\\&\leq\frac{C}{k}\left[\|
\v\|^3_{\L^3(\mathcal{O})}+
\big|z(\theta_{t}\omega)
\big|^3\|\h\|^3_{\L^3(\mathcal{O})}\right] \nonumber\\&\leq\frac{C}{k}
\left[\|\v\|^2_{\H}\|\v\|_{\V}
+\left|z(\theta_{t}\omega)
\right|^3\right]\nonumber\\&\leq\frac{C}{k}
\left[\|\v\|^4_{\H}+\|\v\|^2_{\V}
+\left|z(\theta_{t}\omega)\right|^4+1\right],
	\end{align}
	where we have used Gagliardo-Nirenberg's and Young's inequalities. Using integration by parts, divergence free condition, \eqref{poin} and \eqref{p-value-N}, we get
	\begin{align}\label{ep9-N} -I_4(k,t)&=2\int_{\mathcal{O}}p\uprho\left(\frac{|x|^2}{k^2}\right)\uprho'\left(\frac{|x|^2}{k^2}\right)\frac{2}{k^2}(x\cdot\v)\d x\nonumber\\&\leq\frac{C}{k} \int\limits_{\mathcal{O}\cap\{k\leq|x|\leq \sqrt{2}k\}}\left|p\right|\left|\v\right|\d x\nonumber\\&\leq \frac{C}{k}\bigg[\|\v+\h z(\theta_{t}\omega)\|^2_{\L^4(\mathcal{O})}
\|\v\|_{\H}\bigg]
\nonumber\\&\leq \frac{C}{k}\bigg[\|\v\|^2_{\L^4(\mathcal{O})}
\|\v\|_{\H}+|z(\theta_{t}\omega)|^2\|
\v\|_{\H}\bigg]\nonumber\\&\leq \frac{C}{k}\bigg[\|\v\|_{\V}\|\v\|^2_{\H}
+|z(\theta_{t}\omega)|^4+\|\v\|^2_{\H}\bigg]\nonumber\\&\leq \frac{C}{k}\bigg[\|\v\|^4_{\H}+\|\v\|^2_{\V}
+\left|z(\theta_{t}\omega)\right|^4\bigg],
	\end{align}
	where we have used \eqref{poin}, Gagliardo-Nirenberg's and Young's inequalities. Finally, we estimate the remaining terms of \eqref{ep1-N} by using Hypothesis \ref{AonH-N}, H\"older's and Young's inequalities as follows:
	\begin{align} &\left|I_3(k,t)\right|\nonumber\\&\leq\left|z(\theta_{t}\omega)\right|\left|b\left(\uprho\left(\frac{|x|^2}{k^2}\right)(\v+\h z(\theta_{t}\omega)),\h,\uprho\left(\frac{|x|^2}{k^2}\right)(\v+\h z(\theta_{t}\omega))\right)\right|\nonumber\\&\quad+2\left|z(\theta_{t}\omega)\right|\left|\int_{\mathcal{O}}\uprho\left(\frac{|x|^2}{k^2}\right)\uprho'\left(\frac{|x|^2}{k^2}\right)\Big[\frac{2x}{k^2}\cdot(\v+\h z(\theta_{t}\omega))\Big]\big[(\v+\h z(\theta_{t}\omega))\cdot \h\big]\d x\right|\nonumber\\&\leq{\aleph}\left|z(\theta_{t}\omega)\right|\int_{\mathcal{O}}\uprho^2\left(\frac{|x|^2}{k^2}\right)|\v+\h z(\theta_{t}\omega)|^2\d x +\frac{C}{k}\left|z(\theta_{t}\omega)\right|\int\limits_{\mathcal{O}\cap\{k\leq|x|\leq \sqrt{2}k\}}\left|\v+\h z(\theta_{t}\omega)\right|^2\left|\h\right|\d x\nonumber\\&\leq2{\aleph}\left|z(\theta_{t}\omega)\right|\int_{\mathcal{O}}\left|\uprho\left(\frac{|x|^2}{k^2}\right)\v\right|^2\d x+2{\aleph}\left|z(\theta_{t}\omega)\right|^3\int\limits_{\mathcal{O}_k^{c}}|\h|^2\d x +\frac{C}{k}\bigg[\left|z(\theta_{t}\omega)\right|\|\v\|^2_{\L^4(\mathcal{O})}+\left|z(\theta_{t}\omega)\right|^3\bigg]\nonumber\\&\leq2{\aleph}\left|z(\theta_{t}\omega)\right|\int_{\mathcal{O}}\left|\uprho\left(\frac{|x|^2}{k^2}\right)\v\right|^2\d x+2{\aleph}\left|z(\theta_{t}\omega)\right|^3\int\limits_{\mathcal{O}\cap\{|x|\geq k\}}|\h|^2\d x \nonumber\\&\quad+\frac{C}{k}\bigg[\|\v\|^4_{\H}+\|\v\|^2_{\V}+\left|z(\theta_{t}\omega)\right|^4+1\bigg],
	\end{align}
	and
	\begin{align}
		I_5(k,t)\leq\frac{\nu\lambda}{4}\int_{\mathcal{O}} \left|\uprho\left(\frac{|x|^2}{k^2}\right)\v\right|^2 \d x+C\int_{\mathcal{O}}\uprho^2\left(\frac{|x|^2}{k^2}\right)\bigg[|\f|^2+\left|z(\theta_{t}\omega)\right|^2 |\h|^2+\left|z(\theta_{t}\omega)\right|^2|\Delta\h|^2\bigg]\d x.	\label{ep4-N}
	\end{align}
	Combining \eqref{ep1-N}-\eqref{ep4-N}, we get
	\begin{align}\label{ep5-N}
		&\frac{\d}{\d t}\|\v\|^2_{\mathbb{L}^2(\mathcal{O}_k^{c})}+\big[\nu\lambda-4{\aleph}\left|z(\theta_{t}\omega)\right|\big]\|\v\|^2_{\mathbb{L}^2(\mathcal{O}_k^c)} \nonumber\\ &\leq\frac{C}{k} \bigg[\|\v\|^4_{\H}+\|\v\|^2_{\V}+\left|z(\theta_{t}\omega)\right|^{4}+1\bigg]+2{\aleph}\left|z(\theta_{t}\omega)\right|^3\int_{\mathcal{O}\cap\{|x|\geq k\}}\left|\h(x)\right|^{2}\d x\nonumber\\&\quad+C \int_{\mathcal{O}\cap\{|x|\geq k\}}|\f(x)|^2\d x+C\left|z(\theta_{t}\omega)\right|^2 \int_{\mathcal{O}\cap\{|x|\geq k\}}\left[|\h (x)|^2+|\Delta\h (x)|^2\right]\d x.
	\end{align}
	Making use of variation of constant formula to the above inequality \eqref{ep5-N} on $(s-t,s)$ and replacing $\omega$ by $\theta_{-s}\omega$, we find for $s\leq\tau,\ t\geq 0$ and $\omega\in\Omega$,
	\begin{align}\label{ep6-N}
		&\|\v(s,s-t,\theta_{-s}\omega,\v_{0})\|^2_{\mathbb{L}^2(\mathcal{O}_k^{c})} \nonumber\\& \leq e^{-\nu\lambda t+4{\aleph}\int^{0}_{-t}\left|z(\theta_{\upeta}\omega)\right|\d\upeta}\|\v_{0}\|^2_{\H}+\frac{C}{k}\bigg[\underbrace{\int_{s-t}^{s}e^{\nu\lambda(\uprho-s)-4{\aleph}\int^{\uprho}_{s}\left|z(\theta_{\upeta-s}\omega)\right|\d\upeta}\|\v(\uprho,s-t,\theta_{-s}\omega,\v_{0})\|^4_{\H}\d\uprho}_{\widehat{I}_1(t)}\nonumber\\&\quad+\underbrace{\int_{s-t}^{s}e^{\nu\lambda(\uprho-s)-4{\aleph}\int^{\uprho}_{s}\left|z(\theta_{\upeta-s}\omega)\right|\d\upeta}\|\v(\uprho,s-t,\theta_{-s}\omega,\v_{0})\|^2_{\V}\d\uprho}_{\widehat{I}_2(t)}\nonumber\\&\quad+\underbrace{\int_{-t}^{0}e^{\nu\lambda\uprho-4{\aleph}\int^{\uprho}_{0}\left|z(\theta_{\upeta}\omega)\right|\d\upeta}\bigg\{\left|z(\theta_{\uprho}\omega)\right|^{4}+1\bigg\}\d\uprho}_{\widehat{I}_{3}(t)}\bigg]\nonumber\\&\quad+\underbrace{C\int_{-t}^{0}e^{\nu\lambda\uprho-4{\aleph}\int^{\uprho}_{0}\left|z(\theta_{\upeta}\omega)\right|\d\upeta}\left|z(\theta_{\uprho}\omega)\right|^2\d\uprho\left[\int\limits_{\mathcal{O}\cap\{|x|\geq k\}}|\h (x)|^2\d x+\int\limits_{\mathcal{O}\cap\{|x|\geq k\}}|\Delta\h (x)|^2\d x\right]}_{\widehat{I}_4(k,t)}\nonumber\\&\quad+\underbrace{C\int_{-t}^{0}e^{\nu\lambda\uprho-4{\aleph}\int^{\uprho}_{0}\left|z(\theta_{\upeta}\omega)\right|\d\upeta} \int\limits_{\mathcal{O}\cap\{|x|\geq k\}}|\f(x,\uprho+s)|^2\d x\d\uprho}_{\widehat{I}_5(k,t)}.
	\end{align}
From \eqref{AB2-N}, we obtain
\begin{align}\label{ep7-N} &\widehat{I}_1(t)\nonumber\\&\leq\int_{s-t}^{s}e^{\nu\lambda(\uprho-s)-4{\aleph}\int^{\uprho}_{s}\left|z(\theta_{\upeta-s}\omega)\right|\d\upeta}\bigg[e^{-\nu\lambda(\uprho-s+t)+4{\aleph}\int_{-t}^{\uprho-s}\left|z(\theta_{\upeta}\omega)\right|\d\upeta}\|\v_{0}\|^2_{\H}\nonumber\\&\quad+\widehat{R}_4\int_{-t}^{\uprho-s}e^{\nu\lambda(\uprho_1+s-\uprho)-4{\aleph}\int^{\uprho_1}_{\uprho-s}\left|z(\theta_{\upeta}\omega)\right|\d\upeta}\bigg\{\|\f(\uprho_1+s)\|^2_{\H}+\left|z(\theta_{\uprho_1}\omega)\right|^{3}+1\bigg\}\d\uprho_1\bigg]^2\d\uprho\nonumber\\&\leq C \int_{s-t}^{s}e^{\frac{\nu\lambda}{4}(\uprho-s)-4{\aleph}\int_{\uprho-s}^{0}\left|z(\theta_{\upeta}\omega)\right|\d\upeta}\d\uprho\cdot e^{-\frac{3\nu\lambda}{4}t+8{\aleph}\int_{-t}^{0}\left|z(\theta_{\upeta}\omega)\right|\d\upeta}\|\v_{0}\|^4_{\H}\nonumber\\&\quad+\int_{s-t}^{s}e^{\frac{\nu\lambda}{3}(\uprho-s)-4{\aleph}\int_{\uprho-s}^{0}\left|z(\theta_{\upeta}\omega)\right|\d\upeta}\d\uprho\nonumber\\&\quad\times\bigg(\int_{-\infty}^{0}e^{\frac{\nu\lambda}{3}\uprho_1+4{\aleph}\int_{\uprho_1}^{0}\left|z(\theta_{\upeta}\omega)\right|\d\upeta}\bigg\{\|\f(\uprho_1+s)\|^2_{\H}+\left|z(\theta_{\uprho_1}\omega)\right|^{3}+1\bigg\}\d\uprho_1\bigg)^2\nonumber\\&\leq C \int_{-\infty}^{0}e^{\frac{\nu\lambda}{4}\uprho-4{\aleph}\int_{\uprho}^{0}\left|z(\theta_{\upeta}\omega)\right|\d\upeta}\d\uprho\cdot \left[e^{-\frac{3\nu\lambda}{8}t+4{\aleph}\int_{-t}^{0}\left|z(\theta_{\upeta}\omega)\right|\d\upeta}\|\v_{0}\|^2_{\H}\right]^2+\int_{-\infty}^{0}e^{\frac{\nu\lambda}{3}\uprho-4{\aleph}\int_{\uprho}^{0}\left|z(\theta_{\upeta}\omega)\right|\d\upeta}\d\uprho\nonumber\\&\quad\times\bigg(\int_{-\infty}^{0}e^{\frac{\nu\lambda}{3}\uprho_1+4{\aleph}\int_{\uprho_1}^{0}\left|z(\theta_{\upeta}\omega)\right|\d\upeta}\bigg\{\|\f(\uprho_1+s)\|^2_{\H}+\left|z(\theta_{\uprho_1}\omega)\right|^{3}+1\bigg\}\d\uprho_1\bigg)^2\nonumber\\&:=\widehat{I}_{11}(t)+\widehat{I}_{12}(t).
\end{align}
It follows from \eqref{ep6-N} and \eqref{ep7-N} that
	\begin{align}\label{ep8-N}
	&\|\v(s,s-t,\theta_{-s}\omega,\v_{0})\|^2_{\mathbb{L}^2(\mathcal{O}_k^{c})} \nonumber\\& \leq e^{-\nu\lambda t+4{\aleph}\int^{0}_{-t}\left|z(\theta_{\upeta}\omega)\right|\d\upeta}\|\v_{0}\|^2_{\H}+\frac{C}{k}\bigg[\widehat{I}_{11}(t)+\widehat{I}_{12}(t)+\widehat{I}_{2}(t)+\widehat{I}_{3}(t)\bigg]+\widehat{I}_{4}(k,t)+\widehat{I}_{5}(k,t).
\end{align}
	Now, using the fact that $\h\in\D(\A)$, \eqref{f3-N}, the definition of backward temperedness \eqref{BackTem}, \eqref{Z3}, \eqref{Z6}, \eqref{IRAS2-N} and Lemma \ref{Absorbing-N}, one can  complete the proof.
\end{proof}
The following Lemma provides the backward flattening estimates for the solution of the system \eqref{2-A}. For each $k\geq1$, we let
\begin{align}\label{varrho_k}
	\varrho_k(x):= 1-\uprho\left(\frac{|x|^2}{k^2}\right),  \ \ x\in\mathcal{O}.
\end{align}
Let $\bar{\v}:=\varrho_k\v$ for $\v:=\v(s,s-t,\omega,\v_{\tau})\in\H$. Then $\bar{\v}\in\L^2(\mathcal{O}_{\sqrt{2}k})$, which has the orthogonal decomposition:
\begin{align}\label{DirectProd}
	\bar{\v}=\P_{i}\bar{\v}\oplus(\I-\P_{i})\bar{\v}=:\bar{\v}_{i,1}+\bar{\v}_{i,2},  \ \ \text{ for each } \ i\in\N,
\end{align}
where, $\P_i:\L^2(\mathcal{O}_{\sqrt{2}k})\to\H_{i}:=\mathrm{span}\{e_1,e_2,\cdots,e_i\}\subset\L^2(\mathcal{O}_{\sqrt{2}k})$ is a canonical projection and $\{e_m\}_{m=1}^{\infty}$, is a family of eigenfunctions for $-\Delta$ in $\L^2(\mathcal{O}_{\sqrt{2}k})$ with corresponding  eigenvalues $0<\lambda_1\leq\lambda_2\leq\cdots\leq\lambda_m\to\infty$ as $m\to\infty$. We also have that $$\varrho_k \Delta\v=\Delta\bar{\v}-\v\Delta\varrho_k-2\nabla\varrho_k\cdot\nabla\v.$$ Furthermore, for  $\boldsymbol{\psi}\in\H_0^1(\mathcal{O}_{\sqrt{2}k})$, we have
\begin{align}\label{poin-i}
	\mathrm{P}_i\boldsymbol{\psi}&=\sum_{m=1}^{i}(\boldsymbol{\psi},e_m)e_m,\  \A^{1/2}\mathrm{P}_i\boldsymbol{\psi}=\sum_{m=1}^{i}\lambda^{1/2}_j(\boldsymbol{\psi},e_m)e_m,\nonumber\\ (\I-\mathrm{P}_i)\boldsymbol{\psi}&=\sum_{m=i+1}^{\infty}(\boldsymbol{\psi},e_m)e_m,\ \A^{1/2}(\I-\P_i)\boldsymbol{\psi}=\sum_{m=i+1}^{\infty}\lambda_j(\boldsymbol{\psi},e_m)e_m, \nonumber\\
	\|\nabla(\I-\P_i)\boldsymbol{\psi}\|_{\L^2(\mathcal{O}_{\sqrt{2}k})}^2&=\|\A^{1/2}(\I-\P_i)\boldsymbol{\psi}\|_{\L^2(\mathcal{O}_{\sqrt{2}k})}^2=\sum_{m=i+1}^{\infty}\lambda_m^2|(\boldsymbol{\psi},e_m)|^2\nonumber\\&\geq \lambda_{i+1}\sum_{m=i+1}^{\infty}\lambda_m|(\boldsymbol{\psi},e_m)|^2=\lambda_{i+1}\|(\I-\P_i)\boldsymbol{\psi}\|_{\L^2(\mathcal{O}_{\sqrt{2}k})}^2.
\end{align}
\begin{lemma}\label{Flattening-N}
	Suppose that $\f\in\mathrm{L}^2_{\mathrm{loc}}(\R;\H)$, Hypotheses \ref{assumpO} and \ref{AonH-N} are satisfied. Let $(\tau,\omega,D)\in\R\times\Omega\times\mathfrak{D}$ and $k\geq1$ be fixed. Then
	\begin{align}\label{FL-P}
		\lim_{i,t\to+\infty}\sup_{s\leq \tau}\sup_{\v_{0}\in D(s-t,\theta_{-t}\omega)}\|(\I-\P_{i})\bar{\v}(s,s-t,\theta_{-s},\bar{\v}_{0,2})\|^2_{\L^2(\mathcal{O}_{\sqrt{2}k})}=0,
	\end{align}
	where $\bar{\v}_{0,2}=(\I-\P_{i})(\varrho_k\v_{0})$.
\end{lemma}

\begin{proof}
	Multiplying by $\varrho_k$ in the first equation of \eqref{2-A}, we rewrite the equation as:
	\begin{align}\label{FL1}
		&\frac{\d\bar{\v}}{\d t}-\nu\Delta\bar{\v}+\varrho_k\big((\v+\h z)\cdot\nabla\big)(\v+\h z)+\varrho_k\nabla p\nonumber\\&\quad=-\nu\v\Delta\varrho_k-2\nu\nabla\varrho_k\cdot\nabla\v+\varrho_k\f +\sigma\varrho_k\h z+\nu z\varrho_k\Delta\h.
\end{align}
Applying $(\I-\P_i)$ to the equation \eqref{FL1} and taking the inner product of the resulting equation with $\bar{\v}_{i,2}$ in $\L^2(\mathcal{O}_{\sqrt{2}k})$, we get
\begin{align}\label{FL2}
	&\frac{1}{2}\frac{\d}{\d t}\|\bar{\v}_{i,2}\|^2_{\L^2(\mathcal{O}_{\sqrt{2}k})} +\nu\|\nabla\bar{\v}_{i,2}\|^2_{\L^2(\mathcal{O}_{\sqrt{2}k})}\nonumber\\&=-\underbrace{\sum_{q,m=1}^{2}\int_{\mathcal{O}_{\sqrt{2}k}}\left(\I-\P_i\right)\bigg[(v_{q}+h_{q}z(\theta_{t}\omega))\frac{\partial(v_{m}+h_{m}z(\theta_{t}\omega))}{\partial x_q}\left\{\varrho_k(x)\right\}^2(v_{m}+h_{m}z(\theta_{t}\omega))\bigg]\d x}_{:=J_1}\nonumber\\&\quad+\underbrace{z(\theta_{t}\omega)\sum_{q,m=1}^{2}\int_{\mathcal{O}_{\sqrt{2}k}}\left(\I-\P_i\right)\bigg[(v_{q}+h_{q}z(\theta_{t}\omega))\frac{\partial(v_{m}+h_{m}z(\theta_{t}\omega))}{\partial x_q}\left\{\varrho_k(x)\right\}^2h_{m}\bigg]\d x}_{:=J_2}\nonumber\\&\quad-\underbrace{\left\{\nu\big(\v\Delta\varrho_k,\bar{\v}_{i,2}\big)+2\nu\big(\nabla\varrho_k\cdot\nabla\v,\bar{\v}_{i,2}\big)-\big(\varrho_k\f,\bar{\v}_{i,2}\big) -\sigma z\big(\varrho_k\h,\bar{\v}_{i,2}\big)-\nu z\big(\varrho_k\Delta\h,\bar{\v}_{i,2}\big)\right\}}_{:=J_3}\nonumber\\&\quad-\underbrace{\big(\varrho_k(x)\nabla p, \bar{\v}_{i,2}\big)}_{:=J_4}.
\end{align}
Next, we estimate each terms of \eqref{FL2} as follows: Using integration by parts, divergence free condition of $\v(\cdot)$, \eqref{poin-i} (WLOG we assume that $\lambda_{i}\geq1$), H\"older's, Gagliardo-Nirenberg's (Theorem 1, \cite{Nirenberg}) and Young's inequalities, we find
\begin{align}
\left|J_1\right|&=\left|2\int_{\mathcal{O}_{\sqrt{2}k}}\left(\I-\P_i\right)\bigg[\uprho'\left(\frac{|x|^2}{k^2}\right)\frac{x}{k^2}\cdot\left\{\varrho_k(x)\v+\varrho_k(x)\boldsymbol{h}z(\theta_{t}\omega)\right\}|\v+\boldsymbol{h}z(\theta_{t}\omega)|^2\bigg]\d x\right|\nonumber\\&\leq C\left[\|\bar{\v}_{i,2}\|_{\L^3(\mathcal{O}_{\sqrt{2}k})}+\left|z(\theta_{t}\omega)\right|\|(\I-\P_i)(\varrho_k(x)\boldsymbol{h})\|_{\L^3(\mathcal{O}_{\sqrt{2}k})}\right]\|\v+\boldsymbol{h}z(\theta_{t}\omega)\|^2_{\L^3(\mathcal{O})}\nonumber\\&\leq C\bigg[\|\bar{\v}_{i,2}\|^{\frac{2}{3}}_{\L^2(\mathcal{O}_{\sqrt{2}k})}\|\nabla\bar{\v}_{i,2}\|^{\frac{1}{3}}_{\L^2(\mathcal{O}_{\sqrt{2}k})}+\left|z(\theta_{t}\omega)\right|\|(\I-\P_i)(\varrho_k(x)\boldsymbol{h})\|^{\frac{2}{3}}_{\L^2(\mathcal{O}_{\sqrt{2}k})}\nonumber\\&\qquad\times\|\nabla(\I-\P_i)(\varrho_k(x)\boldsymbol{h})\|^{\frac{1}{3}}_{\L^2(\mathcal{O}_{\sqrt{2}k})}\bigg]\|\v+\boldsymbol{h}z(\theta_{t}\omega)\|^{\frac{4}{3}}_{\H}\|\v+\boldsymbol{h}z(\theta_{t}\omega)\|^{\frac{2}{3}}_{\V}\nonumber\\&\leq C\lambda^{-1/3}_{i+1}\left[\|\nabla\bar{\v}_{i,2}\|_{\L^2(\mathcal{O}_{\sqrt{2}k})}+\left|z(\theta_{t}\omega)\right|\|\nabla[(\I-\P_i)(\varrho_k(x)\boldsymbol{h})]\|_{\L^2(\mathcal{O}_{\sqrt{2}k})}\right]\nonumber\\&\qquad\times\|\v+\boldsymbol{h}z(\theta_{t}\omega)\|^{\frac{4}{3}}_{\H}\|\v+\boldsymbol{h}z(\theta_{t}\omega)\|^{\frac{2}{3}}_{\V}\nonumber\\&\leq \frac{\nu}{8}\|\nabla\bar{\v}_{i,2}\|^2_{\L^2(\mathcal{O}_{\sqrt{2}k})}+C\left|z(\theta_{t}\omega)\right|^2\|\nabla[(\I-\P_i)(\varrho_k(x)\boldsymbol{h})]\|^2_{\L^2(\mathcal{O}_{\sqrt{2}k})}+C\lambda^{-1}_{i+1}\|\v\|^8_{\H}\nonumber\\&\quad+C\lambda^{-1/2}_{i+1}\|\v\|^2_{\V}+C\lambda^{-1}_{i+1}|z(\theta_{t}\omega)|^8+C\lambda^{-1/3}_{i+1},\label{FL3}\\
\left|J_2\right|&\leq\left|z(\theta_{t}\omega)\right|\big|b\big(\bar{\v}_{i,2}+(\I-\P_i)\left(\varrho_k(x)\boldsymbol{h}z(\theta_{t}\omega)\right),\h,\bar{\v}_{i,2}+(\I-\P_i)\left(\varrho_k(x)\boldsymbol{h}z(\theta_{t}\omega)\right)\big)\big|\nonumber\\&\quad+2\left|z(\theta_{t}\omega)\right|\left|\int_{\mathcal{O}_{\sqrt{2}k}}(\I-\P_i)\bigg[\varrho_k(x)\uprho'\left(\frac{|x|^2}{k^2}\right)\big[\frac{x}{k^2}\cdot(\v+\h z(\theta_{t}\omega))\big]\big[(\v+\h z(\theta_{t}\omega))\cdot \h\big]\bigg]\d x\right|\nonumber\\&\leq \aleph\left|z(\theta_{t}\omega)\right|\|\bar{\v}_{i,2}+(\I-\P_i)\left(\varrho_k(x)\boldsymbol{h}z(\theta_{t}\omega)\right)\|^2_{\L^2(\mathcal{O}_{\sqrt{2}k})}\nonumber\\&\quad+C\left|z(\theta_{t}\omega)\right|\|\bar{\v}_{i,2}+(\I-\P_i)\left(\varrho_k(x)\boldsymbol{h}z(\theta_{t}\omega)\right)\|_{\L^2(\mathcal{O}_{\sqrt{2}k})}\|\v+\boldsymbol{h}z(\theta_{t}\omega)\|_{\L^4(\mathcal{O})}\|\h\|_{\L^4(\mathcal{O})}\nonumber\\&\leq \aleph\left|z(\theta_{t}\omega)\right|\|\bar{\v}_{i,2}+(\I-\P_i)\left(\varrho_k(x)\boldsymbol{h}z(\theta_{t}\omega)\right)\|^2_{\L^2(\mathcal{O}_{\sqrt{2}k})}\nonumber\\&\quad+C\lambda^{-1/4}_{i+1}\left|z(\theta_{t}\omega)\right|\|\nabla\bar{\v}_{i,2}+\nabla\left[(\I-\P_i)\left(\varrho_k(x)\boldsymbol{h}z(\theta_{t}\omega)\right)\right]\|^{1/2}_{\L^2(\mathcal{O}_{\sqrt{2}k})}\|\v+\boldsymbol{h}z(\theta_{t}\omega)\|_{\H}\nonumber\\&\qquad\times\|\v+\boldsymbol{h}z(\theta_{t}\omega)\|^{1/2}_{\V}\nonumber\\&\leq2\aleph\left|z(\theta_{t}\omega)\right|\|\bar{\v}_{i,2}\|^2_{\L^2(\mathcal{O}_{\sqrt{2}k})}+\frac{\nu}{8}\|\nabla\bar{\v}_{i,2}\|^{2}_{\L^2(\mathcal{O}_{\sqrt{2}k})}+C\lambda^{-1/3}_{i+1}\|\v\|^2_{\V}+C\lambda^{-1/3}_{i+1}\|\v\|^8_{\H}\nonumber\\&\quad+C\left|z(\theta_{t}\omega)\right|^3\|(\I-\P_i)\left(\varrho_k(x)\boldsymbol{h}\right)\|^2_{\L^2(\mathcal{O}_{\sqrt{2}k})}+C|z(\theta_{t}\omega)|^2\|\nabla\left[(\I-\P_i)\left(\varrho_k(x)\boldsymbol{h}\right)\right]\|^{2}_{\L^2(\mathcal{O}_{\sqrt{2}k})}\nonumber\\&\quad+C\lambda^{-1/3}_{i+1}|z(\theta_{t}\omega)|^8+C\lambda^{1/3}_{i+1},\label{FL4}
\\
\left|J_3\right|&\leq C\bigg[\|\v\|_{\H}+ \|\v\|_{\V}+\|\f\|_{\H}+\left|z(\theta_{t}\omega)\right|\bigg]\|\bar{\v}_{i,2}\|_{\L^2(\mathcal{O}_{\sqrt{2}k})}\nonumber\\&\leq C\lambda^{-1/2}_{i+1}\bigg[\|\v\|_{\V}+\|\f\|_{\H}+\left|z(\theta_{t}\omega)\right|\bigg]\|\nabla\bar{\v}_{i,2}\|_{\L^2(\mathcal{O}_{\sqrt{2}k})}\nonumber\\&\leq\frac{\nu}{8}\|\nabla\bar{\v}_{i,2}\|^2_{\L^2(\mathcal{O}_{\sqrt{2}k})}+ C\lambda^{-1}_{i+1}\bigg[\|\v\|^2_{\V}+\|\f\|^2_{\H}\bigg]+C\lambda^{-1}_{i+1}\left|z(\theta_{t}\omega)\right|^8+C\lambda^{-1}_{i+1},\label{FL5}
\\
\left|J_4\right|&=\left|\int_{\mathcal{O}_{\sqrt{2}k}}(\I-\P_i)\bigg[\nabla p\left\{\varrho_k(x)\right\}^2\v\bigg]\d x\right|=2\left|\int_{\mathcal{O}_{\sqrt{2}k}}(\I-\P_i)\bigg[\uprho'\left(\frac{|x|^2}{k^2}\right)p\varrho_k(x)\v\bigg]\d x\right|\nonumber\\&\leq C\|p\|_{\mathrm{L}^2(\mathcal{O})}\|\bar{\v}_{i,2}\|_{\L^2(\mathcal{O}_{\sqrt{2}k})}\nonumber\\&\leq C\lambda^{-1/4}_{i+1}\|\v+\boldsymbol{h}z(\theta_{t}\omega)\|^2_{\L^4(\mathcal{O})}\|\bar{\v}_{i,2}\|^{1/2}_{\L^2(\mathcal{O}_{\sqrt{2}k})}\|\nabla\bar{\v}_{i,2}\|^{1/2}_{\L^2(\mathcal{O}_{\sqrt{2}k})}\nonumber\\&\leq C\lambda^{-1/4}_{i+1}\|\v+\boldsymbol{h}z(\theta_{t}\omega)\|_{\H}\|\v+\boldsymbol{h}z(\theta_{t}\omega)\|_{\V}\|\v\|^{1/2}_{\H}\|\nabla\bar{\v}_{i,2}\|^{1/2}_{\L^2(\mathcal{O}_{\sqrt{2}k})}\nonumber\\&\leq\frac{\nu}{8}\|\nabla\bar{\v}_{i,2}\|^2_{\L^2(\mathcal{O}_{\sqrt{2}k})}+C\lambda^{-1/4}_{i+1}\|\v\|^2_{\V}+C\lambda^{-1/2}_{i+1}\|\v\|^8_{\H}+C\lambda^{-1/2}_{i+1}|z(\theta_{t}\omega)|^8+C\lambda^{-1/2}_{i+1}+C\lambda^{-1/6}_{i+1},\label{FL6}
\end{align}
where we have used Hypothesis \ref{AonH-N} and equation \eqref{p-value-N} in \eqref{FL4} and \eqref{FL6}, respectively. Now, combining \eqref{FL2}-\eqref{FL6} and using \eqref{poin} in the resulting inequality, we arrive at
\begin{align}\label{FL7}
	&\frac{\d}{\d t}\|\bar{\v}_{i,2}\|^2_{\L^2(\mathcal{O}_{\sqrt{2}k})} +\left(\nu\lambda-4\aleph\left|z(\theta_{t}\omega)\right|\right)\|\bar{\v}_{i,2}\|^2_{\L^2(\mathcal{O}_{\sqrt{2}k})} \nonumber\\&\leq \I_1(i)\left|z(\theta_{t}\omega)\right|^2+\I_2(i)\left|z(\theta_{t}\omega)\right|^3+\I_3(i)\|\v\|^8_{\H}+\I_4(i)\|\v\|^2_{\V}+\I_3(i)|z(\theta_{t}\omega)|^8+\I_5(i)\|\f\|^2_{\H}+\I_{6}(i),
\end{align}
where
\begin{align*}
	&\I_1(i)=C\|(\I-\P_i)\left(\varrho_k(x)\boldsymbol{h}\right)\|^2_{\L^2(\mathcal{O}_{\sqrt{2}k})},\ \ \ \I_2(i)=C\|\nabla[(\I-\P_i)(\varrho_k(x)\boldsymbol{h})]\|^2_{\L^2(\mathcal{O}_{\sqrt{2}k})},\nonumber\\&\I_3(i)=C\left[\lambda^{-1/3}_{i+1}+\lambda^{-1/2}_{i+1}+\lambda^{-1}_{i+1}\right],\ \ \ \ \ \ \ \  \I_4(i)=C\left[\lambda^{-1/4}_{i+1}+\lambda^{-1/3}_{i+1}+\lambda^{-1/2}_{i+1}+\lambda^{-1}_{i+1}\right],\nonumber\\&\I_5(i)=C\lambda^{-1}_{i+1}\ \text{ and }\ \I_6(i)=C\left[\lambda^{-1/6}_{i+1}+\lambda^{-1/3}_{i+1}+\lambda^{-1/2}_{i+1}+\lambda^{-1}_{i+1}\right].
\end{align*}
Due to the fact that $\h\in\D(\A)$ and $\lambda_i\to+\infty$ as $i\to+\infty$, we deduce that
\begin{align}\label{i_convergence}
	\lim_{i\to+\infty}\I_1(i)=\lim_{i\to+\infty}\I_2(i)=\lim_{i\to+\infty}\I_3(i)=\lim_{i\to+\infty}\I_4(i)=\lim_{i\to+\infty}\I_5(i)=\lim_{i\to+\infty}\I_6(i)=0.
\end{align} In the view of  variation of constant formula in \eqref{FL7}, we find
\begin{align}\label{FL8}
	&\|(\I-\P_{i})\bar{\v}(s,s-t,\theta_{-s},\bar{\v}_{0,2})\|^2_{\L^2(\mathcal{O}_{\sqrt{2}k})}\nonumber\\&\leq e^{-\nu\lambda t+4{\aleph}\int^{0}_{-t}\left|z(\theta_{\upeta}\omega)\right|\d\upeta}\|(\I-\P_i)(\varrho_k\v_{0})\|^2_{\L^2(\mathcal{O}_{\sqrt{2}k})}+\I_6(i)\underbrace{\int_{-t}^{0}e^{\nu\lambda\uprho-4{\aleph}\int^{\uprho}_{0}\left|z(\theta_{\upeta}\omega)\right|\d\upeta}\d\uprho}_{L_7(t)}\nonumber\\&\quad+\I_4(i)\underbrace{\int_{s-t}^{s}e^{\nu\lambda(\uprho-s)-4{\aleph}\int^{\uprho}_{s}\left|z(\theta_{\upeta-s}\omega)\right|\d\upeta}\|\v(\uprho,s-t,\theta_{-s}\omega,\v_{0})\|^2_{\V}\d\uprho}_{L_1(s,t)}\nonumber\\&\quad+\I_3(i)\underbrace{\int_{s-t}^{s}e^{\nu\lambda(\uprho-s)-4{\aleph}\int^{\uprho}_{s}\left|z(\theta_{\upeta-s}\omega)\right|\d\upeta}\|\v(\uprho,s-t,\theta_{-s}\omega,\v_{0})\|^8_{\H}\d\uprho}_{L_2(s,t)}\nonumber\\&\quad+\I_1(i)\underbrace{\int_{-t}^{0}e^{\nu\lambda\uprho-4{\aleph}\int^{\uprho}_{0}\left|z(\theta_{\upeta}\omega)\right|\d\upeta}\left|z(\theta_{\uprho}\omega)\right|^{2}\d\uprho}_{L_3(t)}\ +\ \I_2(i)\underbrace{\int_{-t}^{0}e^{\nu\lambda\uprho-4{\aleph}\int^{\uprho}_{0}\left|z(\theta_{\upeta}\omega)\right|\d\upeta}\left|z(\theta_{\uprho}\omega)\right|^{3}\d\uprho}_{L_4(t)}\nonumber\\&\quad+\I_3(i)\underbrace{\int_{-t}^{0}e^{\nu\lambda\uprho-4{\aleph}\int^{\uprho}_{0}\left|z(\theta_{\upeta}\omega)\right|\d\upeta}\left|z(\theta_{\uprho}\omega)\right|^{8}\d\uprho}_{L_5(t)}\ + \ \I_5(i)\underbrace{\int_{-t}^{0}e^{\nu\lambda\uprho-4{\aleph}\int^{\uprho}_{0}\left|z(\theta_{\upeta}\omega)\right|\d\upeta}\|\f(\uprho+s)\|^2_{\H}\d\uprho}_{L_6(s,t)}.
\end{align}
It implies from \eqref{Z3}, \eqref{G3}, \eqref{AB1-N}, \eqref{Z6} and \eqref{IRAS2-N} that
\begin{equation}\label{FL9}
	\left\{
	\begin{aligned}
		\sup_{s\leq \tau}L_1(s,t)<+\infty,\	\sup_{s\leq \tau}L_6(s,t)<+\infty,\ \ \ \ \ \ \ \ \ \ \ \ \ \ \ \ \\ L_3(t)<+\infty,\ L_4(t)<+\infty,\ L_5(t)<+\infty\ \text{ and }\ L_7(t)<+\infty,
	\end{aligned}
	\right.
\end{equation}
for sufficiently large $t>0$. Moreover, similar arguments as in \eqref{ep8-N} provide
\begin{align}\label{FL10}
	\sup\limits_{s\leq \tau}L_2(s,t)<+\infty.
\end{align} Further,
\begin{align}\label{FL11}
	\|(\I-\P_i)(\varrho_k\v_{0})\|^2_{\L^2(\mathcal{O}_{\sqrt{2}k})}\leq C\|\v_{0}\|^2_{\H},
\end{align}
for all $\v_{0}\in D(s-t,\theta_{-t}\omega)$ and $s\leq\tau$. Now, using the definition of backward temperedness \eqref{BackTem}, \eqref{Z3}, \eqref{G3}, \eqref{Z6}, Lemma \ref{Absorbing-N} and \eqref{i_convergence}, \eqref{FL9}-\eqref{FL11} in \eqref{FL8}, we obtain \eqref{FL-P}, as desired, which  completes the proof.
\end{proof}

\subsection{Proof of Theorem \ref{MT1-N}}\label{thm1.4}
In this subsection, we demonstrate the main result of this section, that is, the existence of $\mathfrak{D}$-pullback random attractors and their asymptotic autonomy for the solution of the system \eqref{SNSE-A}. For the existence of a unique random attractor for autonomous 2D SNSE driven by additive noise on Poincar\'e domains (bounded or unbounded), we  refer to \cite{BCLLLR}. The proof of this theorem is  divided into following seven steps:
	\vskip 2mm
	\noindent
	\textbf{Step I:} \textit{$\mathfrak{D}$-pullback time-semi-uniform asymptotic compactness of $\Phi$.} It is enough to prove that for each $(\tau,\omega,D)\in\R\times\Omega\times\mathfrak{D}$, arbitrary sequences $s_n\leq\tau$, $\tau_{n}\to+\infty$ and $\v_{0,n}\in D(s_n-t_n,\theta_{-t_n}\omega)$,  the sequence $$\v_n=\v(s_n,s_n-t_n,\theta_{-s_n}\omega,\v_{0,n})$$ is pre-compact. Let $E_{N}=\{\v_n:n\geq N\}, \ N=1,2,\ldots.$ In order to prove the pre-compactness of the sequence $\v_n$, it is enough to prove that the Kuratowski measure $\kappa_{\H}(E_{N})\to0$ and $N\to+\infty$, (cf. Lemma \ref{K-BAC}).
	
	For each $\eta>0$, by Lemma \ref{largeradius-N}, there exists $\mathcal{N}_1\in\N$ and $K\geq1$ such that
	\begin{align}\label{MTA1}
		\|\v_n\|_{\L^2(\mathcal{O}^{c}_{K})}\leq\eta, \ \text{ for all } \ n\geq \mathcal{N}_1,
	\end{align}
	where $\mathcal{O}^{c}_{K}=\mathcal{O}\backslash\mathcal{O}_{K}$ and  $\mathcal{O}_{k}=\{x\in\mathcal{O}:|x|\leq k\}.$  By Lemma \ref{Flattening-N}, there exist $i\in\N$ and $\mathcal{N}_2\geq\mathcal{N}_1$ such that
	\begin{align}\label{MTA2}
		\|(\I-\P_i)(\varrho_{K}\v_n)\|_{\L^2(\mathcal{O}_{\sqrt{2}K})}\leq\eta, \ \text{ for all } \ n\geq\mathcal{N}_2.
	\end{align}

Now, Lemma \ref{Absorbing-N} gives us that the set $E_{\mathcal{N}_2}$ is bounded in $\H$. Then, the set $\{\varrho_{K}\v_n:n\geq\mathcal{N}_2\}$ is bounded in $\L^2(\mathcal{O}_{\sqrt{2}K})$. Hence, by the finite-dimensional range of $\P_i$, $\P_{i}\{\varrho_{K}\v_n:n\geq\mathcal{N}_2\}$ is pre-compact in $\L^2(\mathcal{O}_{\sqrt{2}K})$, from which we conclude that
\begin{align}\label{MTA3}
	\kappa_{\L^2(\mathcal{O}_{\sqrt{2}K})}\left(\P_{i}\{\varrho_{K}\v_n:n\geq\mathcal{N}_2\}\right)=0.
\end{align}
It follows from \eqref{MTA2}-\eqref{MTA3} and Theorem 1.4, \cite{Rakocevic} that
	\begin{align}\label{MTA4}
		&\kappa_{\L^2(\mathcal{O}_{\sqrt{2}K})}\left(\{\varrho_{K}\v_n:n\geq\mathcal{N}_2\}\right)\nonumber\\&\leq\kappa_{\L^2(\mathcal{O}_{\sqrt{2}K})}\left(\P_{i}\{\varrho_{K}\v_n:n\geq\mathcal{N}_2\}\right)+\kappa_{\L^2(\mathcal{O}_{\sqrt{2}K})}\left((\I-\P_{i})\{\varrho_{K}\v_n:n\geq\mathcal{N}_2\}\right)\leq2\eta.
	\end{align}
Since $\varrho_{K}\v_n=\v_n$ on $\mathcal{O}_{K}$, we get from \eqref{MTA4} and Lemma 1.2, \cite{Rakocevic}  that
\begin{align}\label{MTA5}
	\kappa_{\L^2(\mathcal{O}_{K})}(E_{\mathcal{N}_2})&=\kappa_{\L^2(\mathcal{O}_{K})}\{\varrho_{K}\v_n:n\geq\mathcal{N}_2\}\leq \kappa_{\L^2(\mathcal{O}_{\sqrt{2}K})}\{\varrho_{K}\v_n:n\geq\mathcal{N}_2\}\leq2\eta.
\end{align}
Since $E_{\mathcal{N}_2}\subset E_{\mathcal{N}_1}$, it implies from \eqref{MTA1} and \eqref{MTA5} that
\begin{align*}
	\kappa_{\H}(E_{\mathcal{N}_2})\leq\kappa_{\L^2(\mathcal{O}_{K})}(E_{\mathcal{N}_2})+\kappa_{\L^2(\mathcal{O}^{c}_{K})}(E_{\mathcal{N}_1})\leq 3\eta,
\end{align*}
which shows that $\Phi$ is time-semi-uniformly asymptotically compact in $\H$.
	\vskip 2mm
\noindent
\textbf{Step II:} \textit{$\mathfrak{B}$-pullback asymptotically compactness of $\Phi$.} It has been proved in \cite{BCLLLR,BL},  and we are omitting the proof here. Moreover, one can prove the $\mathfrak{B}$-pullback asymptotically compactness of $\Phi$ by using similar arguments as in Step I.
	\vskip 2mm
\noindent
\textbf{Step III:} \textit{$\mathfrak{D}$-pullback attractor $\mathcal{A}(\tau,\omega)$.} Proposition \ref{IRAS-N} ((i) part) and Step I ensure us that $\Phi$ has $\mathfrak{D}$-pullback absorbing set and $\Phi$ is $\mathfrak{D}$-pullback asymptotically compact, respectively. Hence, by the abstract theory established in \cite{SandN_Wang}, $\Phi$ has a unique $\mathfrak{D}$-pullback attractor $\mathcal{A}$ which is given by
\begin{align}\label{A1}
	\mathcal{A}=\cap_{t_0>0}\overline{\cup_{t\geq t_0}\Phi(t,\tau-t,\theta_{-t}\omega)\mathcal{R}(\tau-t,\theta_{-t}\omega)}^{\H}.
\end{align}
However, we remark that the $\mathscr{F}$-measurability of $\mathcal{A}$ is unknown, therefore we are saying $\mathcal{A}$ is a $\mathfrak{D}$-pullback attractor instead of $\mathfrak{D}$-pullback random attractor.
	\vskip 2mm
\noindent
\textbf{Step IV:} \textit{$\mathfrak{B}$-pullback attractor $\widetilde{\mathcal{A}}(\tau,\omega)$.} Proposition \ref{IRAS-N} ((ii) part) and Step II ensure us that $\Phi$ has $\mathfrak{B}$-pullback random absorbing set and $\Phi$ is $\mathfrak{B}$-pullback asymptotically compact, respectively. Hence, by the abstract theory established in \cite{SandN_Wang}, $\Phi$ has a unique $\mathfrak{D}$-pullback random attractor $\mathcal{A}$ which is given by
\begin{align}\label{A2}
	\widetilde{\mathcal{A}}=\cap_{t_0>0}\overline{\cup_{t\geq t_0}\Phi(t,\tau-t,\theta_{-t}\omega)\widetilde{\mathcal{R}}(\tau-t,\theta_{-t}\omega)}^{\H}.
\end{align}
	\vskip 2mm
\noindent
\textbf{Step V:} \textit{Time-semi-uniformly compactness of $\mathcal{A}(\tau,\omega)$.} We prove that $\cup_{s\leq\tau}\mathcal{A}(s,\omega)$ is pre-compact in $\H$. Let $\{\u_n\}_{n=1}^{\infty}$ be an arbitrary sequence extracted from $\cup_{s\leq\tau}\mathcal{A}(s,\omega)$. Then, we can find a sequence $s_n\leq\tau$ such that $\u_n\in\mathcal{A}(s_n,\omega)$ for each $n\in\N$. Now, for the sequence $t_n\to\infty$, by the invariance property of $\mathcal{A}$ we have $\u_n\in\Phi(t_n,s_n-t_n,\theta_{-t_{n}}\omega)\mathcal{A}(s_n-t_n,\theta_{-t_{n}}\omega)$. It implies that we can find $\u_{0,n}\in\mathcal{A}(s_n-t_n,\theta_{-t_{n}}\omega)$ such that $\u_n=\Phi(t_n,s_n-t_n,\theta_{-t_{n}}\omega,\u_{0,n})$. Here, $\u_{0,n}\in\mathcal{A}(s_n-t_n,\theta_{-t_{n}}\omega)\subseteq\mathcal{R}(s_n-t_n,\theta_{-t_{n}}\omega)$ with $s_n\leq\tau$ and $\mathcal{R}\in\mathfrak{D}$, and  it follows from  the $\mathfrak{D}$-pullback time-semi-uniform asymptotic compactness of $\Phi$ that the sequence $(\u_n)$ is pre-compact in $\H$. Hence, $\cup_{s\leq\tau}\mathcal{A}(s,\omega)$ is pre-compact in $\H$.
	\vskip 2mm
\noindent
\textbf{Step VI:} \textit{$\mathcal{A}(\tau,\omega)=\widetilde{\mathcal{A}}(\tau,\omega)$. This  implies that $\Phi$ has a unique pullback \textbf{random} attractor which is time-semi-uniformly compact in $\H$.} Let us fix $(\tau,\omega)\in\R\times\Omega$. Since, by Proposition \ref{IRAS-N}, $\mathcal{R}(\tau,\omega)\supseteq\widetilde{\mathcal{R}}(\tau,\omega)$, it follows from \eqref{A1} and \eqref{A2} that $\mathcal{A}(\tau,\omega)\supseteq\widetilde{\mathcal{A}}(\tau,\omega)$. At the same time, since $\mathcal{A}\in\mathfrak{B}\subseteq\mathfrak{D}$, the invariance property of $\mathcal{A}$ and the attraction property of $\widetilde{\mathcal{A}}$ imply that
\begin{align*}
	\text{dist}_{\H}(\mathcal{A}(\tau,\omega),\widetilde{\mathcal{A}}(\tau,\omega))=\text{dist}_{\H}(\Phi(t,\tau-t,\theta_{-t}\omega)\mathcal{A}(\tau-t,\theta_{-t}\omega),\widetilde{\mathcal{A}}(\tau,\omega))\to 0\ \text{ as } \ t\to+\infty.
\end{align*}
This indicates that $\mathcal{A}(\tau,\omega)\subseteq\widetilde{\mathcal{A}}(\tau,\omega)$. Hence $\mathcal{A}(\tau,\omega)=\widetilde{\mathcal{A}}(\tau,\omega)$, which, in view of the $\mathscr{F}$-measurability of $\widetilde{\mathcal{A}}(\tau,\omega)$, shows that $\mathcal{A}(\tau,\omega)$ is $\mathscr{F}$-measurable.
	\vskip 2mm
	\noindent
\textbf{Step VII:} \textit{Proofs of \eqref{MT2-N} and \eqref{MT3-N}.} By using {Propositions} \ref{IRAS-N} and \ref{Back_conver-N} and \emph{time-semi-uniformly compactness} of $\mathcal{A}(\tau,\omega)$, and applying similar arguments as in the proof of Theorem 5.2, \cite{CGTW}, one can complete the proof. Since, the arguments are similar to the proof of Theorem 5.2 in \cite{CGTW}, we are omitting it here.

\section{Asymptotically autonomous robustness of random attractors of \eqref{1}: multiplicative noise}\label{sec4}\setcounter{equation}{0}
In this section, we consider 2D SNSE equations driven by a multiplicative white noise ($S(\u)=\u$ in \eqref{1}) and establish the existence and asymptotic autonomy of $\mathfrak{D}$-pullback random attractors. Let us define $\v(t,\tau,\omega,\v_{\tau})=e^{-z(\theta_{t}\omega)}\u(t,\tau,\omega,\u_{\tau})\ \text{ with  }\  \v_{\tau}=e^{-z(\theta_{\tau}\omega)}\u_{\tau},$ where $z$ satisfies \eqref{OU2} and $\u(\cdot)$ is the solution of \eqref{1} with $S(\u)=\u$. Then $\v(\cdot)$ satisfies:
		\begin{equation}\label{2-M}
		\left\{
		\begin{aligned}
			\frac{\d\v(t)}{\d t}-\nu \Delta\v(t)&+e^{z(\theta_{t}\omega)}(\v(t)\cdot\nabla)\v(t)+e^{-z(\theta_{t}\omega)}\nabla p\\&=\f(t)e^{-z(\theta_{t}\omega)} +\sigma z(\theta_t\omega)\v(t),  \ \ \ \ \ \ \  \text{ in }\  \mathcal{O}\times(\tau,\infty), \\ \nabla\cdot\v&=0, \hspace{52.5mm} \  \text{ in } \ \ \mathcal{O}\times(\tau,\infty), \\
			\v(x,\tau)&=\v_{0}(x)=e^{-z(\theta_{\tau}\omega)}\u_{0}(x),  \hspace{17mm} \ x\in \mathcal{O} \ \text{ and }\ \tau\in\R,\\
			\v(x,\tau)&=0,\ \ \ \ \ \ \ \  \ \ \ \ \ \ \ \ \qquad\qquad\qquad\qquad \text{ in } \ \ \partial\mathcal{O}\times(\tau,\infty),
		\end{aligned}
		\right.
	\end{equation}
	as well as (projected form)
\begin{equation}\label{CNSE-M}
	\left\{
	\begin{aligned}
		\frac{\d\v(t)}{\d t}+\nu \A\v(t)+e^{z(\theta_{t}\omega)}\B\big(\v(t)\big)&=\f(t) e^{-z(\theta_{t}\omega)} + \sigma z(\theta_t\omega)\v(t) , \quad t> \tau, \tau\in\R ,\\
		\v(x,\tau)&=\v_{0}(x)=e^{-z(\theta_{\tau}\omega)}\u_{0}(x), \hspace{12mm} x\in\mathcal{O},
	\end{aligned}
	\right.
\end{equation}
in $\V^*$. Next, we consider the autonomous SNSE with multiplicative white noise corresponding to the non-autonomous system \eqref{SNSE} with $S(\u)=\u$  as
\begin{equation}\label{A-SNSE-M}
	\left\{
	\begin{aligned}
		\frac{\d\widetilde{\u}(t)}{\d t}+\nu \A\widetilde{\u}(t)+\B(\widetilde{\u}(t))&=\f_{\infty} +\widetilde{\u}(t)\circ\frac{\d \W(t)}{\d t} , \\
		\widetilde{\u}(x,0)&=\widetilde{\u}_{0}(x),	\ x\in \mathcal{O}.
	\end{aligned}
	\right.
\end{equation}
Let $\widetilde{\v}(t,\omega)=e^{-z(\theta_{t}\omega)}\widetilde{\u}(t,\omega)$. Then, the system \eqref{A-SNSE-M} can be written in the following pathwise deterministic system:
\begin{equation}\label{A-CNSE-M}
	\left\{
	\begin{aligned}
		\frac{\d\widetilde{\v}(t)}{\d t}+\nu \A\widetilde{\v}(t)+e^{z(\theta_{t}\omega)}\B\big(\widetilde{\v}(t)\big)&=\f_{\infty} e^{-z(\theta_{t}\omega)} + \sigma z(\theta_t\omega)\widetilde{\v}(t) , \quad t> \tau, \tau\in\R ,\\
		\widetilde{\v}(x,0)&=\widetilde{\v}_{0}(x)=e^{-z(\omega)}\widetilde{\u}_{0}(x), \hspace{13mm} x\in\mathcal{O},
	\end{aligned}
	\right.
\end{equation}
in $\V^*$. The following Lemma shows the well-posedness result for the system \eqref{CNSE-M} which can be proved by a standard Faedo-Galerkin approximation method (cf. \cite{BL}).
\begin{lemma}\label{Soln}
	Suppose that $\f\in\mathrm{L}^2_{\mathrm{loc}}(\R;\H)$. For each $(\tau,\omega,\v_{\tau})\in\R\times\Omega\times\H$, the system \eqref{CNSE-M} has a unique weak solution $\v(\cdot,\tau,\omega,\v_{\tau})\in\mathrm{C}([\tau,+\infty);\H)\cap\mathrm{L}^2_{\mathrm{loc}}(\tau,+\infty;\V)$ such that $\v$ is continuous with respect to the  initial data.
\end{lemma}
The energy inequality in the following Lemma will be used frequently in the rest of the paper.
	\begin{lemma}
		For $\f\in\mathrm{L}^2_{\mathrm{loc}}(\R;\H)$, the solution of \eqref{CNSE-M} satisfies the following inequality:
		\begin{align}\label{EI1}
			\frac{\d}{\d t} \|\v\|^2_{\H}+ \left(\nu\lambda-2\sigma z(\theta_{t}\omega)\right)\|\v\|^2_{\H}+\frac{\nu}{2}\|\v\|^2_{\V} \leq \frac{2e^{2\left|z(\theta_{t}\omega)\right|}}{\nu\lambda}\|\f\|^2_{\H}.
		\end{align}
	\end{lemma}
		\begin{proof}
		From the first equation of the system \eqref{CNSE-M} and \eqref{b0}, we obtain
		
		\begin{align*}
		\frac{1}{2}\frac{\d}{\d t} \|\v\|^2_{\H} +\frac{3\nu}{4}\|\v\|^2_{\V}+\frac{\nu}{4}\|\v\|^2_{\V}&= e^{-z(\theta_{t}\omega)}\left(\f,\v\right)+\sigma z(\theta_{t}\omega)\|\v\|^2_{\H}\nonumber\\&\leq \frac{\nu\lambda}{4}\|\v\|^2_{\H}+\frac{e^{2\left|z(\theta_{t}\omega)\right|}}{\nu\lambda}\|\f\|^2_{\H}+\sigma z(\theta_{t}\omega)\|\v\|^2_{\H}.
 		\end{align*}	
		Now, using \eqref{poin} in the second term on the left hand side of the above inequality, one can conclude the proof.
	\end{proof}

Next result shows the Lusin continuity of mapping with respect to $\omega\in\Omega$ of solution to the system \eqref{CNSE-M} which is taken from the work \cite{KRM} (Proposition 3.4, \cite{KRM}).
\begin{proposition}\label{LusinC}
	Suppose that $\f\in\mathrm{L}^2_{\mathrm{loc}}(\R;\H)$. For each $N\in\N$, the mapping $\omega\mapsto\v(t,\tau,\omega,\v_{\tau})$ $($solution of \eqref{CNSE-M}$)$ is continuous from $(\Omega_{N},d_{\Omega_N})$ to $\H$, uniformly in $t\in[\tau,\tau+T]$ with $T>0$.
\end{proposition}

In view of Lemma \ref{Soln}, we can define a mapping $\Psi:\R^+\times\R\times\Omega\times\H\to\H$ by
\begin{align}\label{Phi}
	\Psi(t,\tau,\omega,\u_{\tau})=\u(t+\tau,\tau,\theta_{-\tau}\omega,\u_{\tau})=e^{z(\theta_{t}\omega)}\v(t+\tau,\tau,\theta_{-\tau}\omega,\v_{\tau}).
\end{align}
The Lusin continuity in Proposition \ref{LusinC} gives the $\mathscr{F}$-measurability of $\Psi$. Consequently, the mapping $\Psi$ defined by \eqref{Phi} is an NRDS on $\H$. The following Proposition demonstrate the backward convergence of NRDS \eqref{Phi} which is adapted from \cite{KRM} (Proposition 3.6, \cite{KRM}).
\begin{proposition}\label{Back_conver}
	Suppose that Hypothesis \ref{Hypo_f-N} is satisfied. Then, the solution $\v$ of the system \eqref{CNSE-M} backward converges to the solution $\widetilde{\v}$ of the system \eqref{A-CNSE-M}, that is,
	\begin{align}
		\lim_{\tau\to -\infty}\|\v(T+\tau,\tau,\theta_{-\tau}\omega,\v_{\tau})-\widetilde{\v}(t,\omega,\widetilde{\v}_0)\|_{\H}=0, \ \ \text{ for all } \ T>0 \ \text{ and } \ \omega\in\Omega,
	\end{align}
whenever $\|\v_{\tau}-\widetilde{\v}_0\|_{\H}\to0$ as $\tau\to-\infty.$
\end{proposition}
Next Lemma is needed to obtain the increasing random absorbing set and the inequality \eqref{AB3} (see below) is  used to prove the backward uniform-tail estimates (Lemma \ref{largeradius}) and the backward flattening estimates (Lemma \ref{Flattening}).
\begin{lemma}\label{Absorbing}
	Suppose that $\f\in\mathrm{L}^2_{\mathrm{loc}}(\R;\H)$. Then, for each $(\tau,\omega,D)\in\R\times\Omega\times\mathfrak{D},$ there exists a time $\mathcal{T}:=\mathcal{T}(\tau,\omega,D)>0$ such that
	\begin{align}\label{AB1}
		&\sup_{s\leq \tau}\sup_{t\geq \mathcal{T}}\sup_{\v_{0}\in D(s-t,\theta_{-t}\omega)}\bigg[\|\v(s,s-t,\theta_{-s}\omega,\v_{0})\|^2_{\H}\nonumber\\&+\frac{\nu}{2}\int_{s-t}^{s}e^{\nu\lambda(\uprho-s)-2\sigma\int^{\uprho}_{s}z(\theta_{\upeta-s}\omega)\d\upeta}\|\v(\uprho,s-t,\theta_{-s}\omega,\v_{0})\|^2_{\V}\d\uprho\bigg]\leq 1+\frac{2}{\nu\lambda}\sup_{s\leq \tau}K(\tau,\omega),
	\end{align}
where $K(\tau,\omega)$ is given by
\begin{align}\label{AB2}
	K(\tau,\omega):=\int_{-\infty}^{0}e^{\nu\lambda\uprho+2|z(\theta_{\uprho}\omega)|+2\sigma\int_{\uprho}^{0}z(\theta_{\upeta}\omega)\d\upeta}\|\f(\uprho+s)\|^2_{\H}\d\uprho.
\end{align}
 Furthermore, for all $\xi> s-t,$ $ t\geq0$ and $\v_{0}\in\H$,
\begin{align}\label{AB3}
		&\|\v(\xi,s-t,\theta_{-s}\omega,\v_{0})\|^2_{\H}+\frac{\nu}{2}\int_{s-t}^{\xi}e^{\nu\lambda(\uprho-\xi)-2\sigma\int^{\uprho}_{\xi}z(\theta_{\upeta-s}\omega)\d\upeta}\|\v(\uprho,s-t,\theta_{-s}\omega,\v_{0})\|^2_{\V}\d\uprho\nonumber\\&\leq e^{-\nu\lambda(\xi-s+t)+2\sigma\int_{-t}^{\xi-s}z(\theta_{\upeta}\omega)\d\upeta}\|\v_{0}\|^2_{\H} + \frac{2}{\nu\lambda}\int\limits_{-t}^{\xi-s}e^{\nu\lambda(\uprho+s-\xi)+2|z(\theta_{\uprho}\omega)|+2\sigma\int_{\uprho}^{\xi-s}z(\theta_{\upeta}\omega)\d\upeta}\|\f(\uprho+s)\|^2_{\H}\d\uprho.
\end{align}
\end{lemma}
\begin{proof}
	See the proof of Lemma 3.7 in \cite{KRM}.
\end{proof}
\begin{proposition}\label{IRAS}
	Suppose that $\f\in\mathrm{L}^2_{\mathrm{loc}}(\R;\H)$, Hypotheses \ref{assumpO} and \ref{Hypo_f-N} are satisfied. For $K(\tau,\omega)$ same as in \eqref{AB2}, we have
	\vskip 2mm
	\noindent
	\emph{(i)} There is an increasing $\mathfrak{D}$-pullback absorbing set $\mathcal{K}$ given by
	\begin{align}\label{IRAS1}
	\mathcal{K}(\tau,\omega):=\left\{\u\in\H:\|\u\|^2_{\H}\leq e^{z(\omega)}\left[1+\frac{2}{\nu\lambda}\sup_{s\leq \tau}K(\tau,\omega)\right]\right\}, \ \text{ for all } \ \tau\in\R,
\end{align}
	Moreover, $\mathcal{K}$ is backward-uniformly tempered with arbitrary rate, that is, $\mathcal{K}\in{\mathfrak{D}}$.
	\vskip 2mm
	\noindent
	\emph{(ii)} There is a $\mathfrak{B}$-pullback \textbf{random} absorbing set $\widetilde{\mathcal{K}}$ given by
	\begin{align}\label{IRAS11}
		\widetilde{\mathcal{K}}(\tau,\omega):=\left\{\u\in\H:\|\u\|^2_{\H}\leq e^{z(\omega)}\left[1+\frac{2}{\nu\lambda}K(\tau,\omega)\right]\right\}, \ \text{ for all } \ \tau\in\R.
	\end{align}
\end{proposition}
\begin{proof}
	(i) Using \eqref{G3}, \eqref{Z5} and \eqref{Z3}, we obtain
	\begin{align}\label{IRAS2}
		\sup_{s\leq \tau}K(s,\omega)&= \sup_{s\leq \tau}\int_{-\infty}^{0}e^{\nu\lambda\uprho+2|z(\theta_{\uprho}\omega)|+2\sigma\int_{\uprho}^{0}z(\theta_{\upeta}\omega)\d\upeta}\|\f(\uprho+s)\|^2_{\H}\d\uprho<\infty.
	\end{align}
	Hence, absorption follows from Lemma \ref{Absorbing}. Due to the fact that $\tau\mapsto \sup_{s\leq\tau}K(\tau,\omega)$ is an increasing function, $\mathcal{K}(\tau,\omega)$ is an increasing $\mathfrak{D}$-pullback absorbing set. Using similar arguments as in \eqref{IRAS3-N}, with the help of \eqref{G3}, \eqref{Z5} and \eqref{Z3}, we deduce
	\begin{align}
		&\lim_{t\to+\infty}e^{-ct}\sup_{s\leq \tau}\|\mathcal{K}(s-t,\theta_{-t}\omega)\|^2_{\H}=0,
	\end{align}
	which gives $\mathcal{K}\in{\mathfrak{D}}$.
	\vskip 2mm
	\noindent
	(ii) Since $\widetilde{\mathcal{K}}\subseteq\mathcal{K}\in\mathfrak{D}\subseteq\mathfrak{B}$ and the mapping $\omega\mapsto K(\tau,\omega)$ is $\mathscr{F}$-measurable,  $\widetilde{\mathcal{R}}$ is a $\mathfrak{B}$-pullback \textbf{random} absorbing set.
\end{proof}

\subsection{Backward uniform-tail estimates and backward flattening estimates}
In this subsection, we prove the backward tail-estimates and backward flattening estimates for the solution of \eqref{2-M}. These estimates help us to prove the time-semi-uniform asymptotic compactness of the solution of \eqref{CNSE-M}. We obtain these estimates by using an appropriate cut-off function. The following Lemma provides the backward tail-estimates for the solution of the system \eqref{2-M}.
\begin{lemma}\label{largeradius}
	Suppose that Hypothesis \ref{Hypo_f-N} is satisfied. Then, for any $(\tau,\omega,D)\in\R\times\Omega\times\mathfrak{D},$ the solution of \eqref{2-M}  satisfies
	\begin{align}\label{ep}
		&\lim_{k,t\to+\infty}\sup_{s\leq \tau}\sup_{\v_{0}\in D(s-t,\theta_{-t}\omega)}\|\v(s,s-t,\theta_{-s}\omega,\v_{0})\|^2_{\mathrm{L}^2(\mathcal{O}^{c}_{k})}=0,
	\end{align}
	where $\mathcal{O}_{k}=\{x\in\mathcal{O}:|x|\leq k\},$ $k\in\mathbb{N}$.
\end{lemma}
\begin{proof}
	Let $\uprho$ be the  smooth function same as defined in \eqref{337}. Similar to \eqref{p-value} and \eqref{p-value-N}, we obtain   from \eqref{2-M}
	\begin{align}\label{p-value-M}
		p=(-\Delta)^{-1}\left[e^{2z(\theta_{t}\omega)}\sum_{i,j=1}^{2}\frac{\partial^2}{\partial x_i\partial x_j}\big(v_iv_j\big)\right],
	\end{align}
	in the weak sense and
	\begin{align}
		\|p\|^2_{\mathrm{L}^2(\mathcal{O})}\leq C e^{4z(\theta_{t}\omega)} \|\v\|^4_{\L^4(\mathcal{O})}\label{p-value-N-M}.
	\end{align}
Taking the inner product to the first equation of \eqref{2-M} with $\uprho^2\left(\frac{|x|^2}{k^2}\right)\v$, we have
	\begin{align}\label{ep1}
		\frac{1}{2} \frac{\d}{\d t} \int_{\mathcal{O}}\uprho^2\left(\frac{|x|^2}{k^2}\right)|\v|^2\d x&= \underbrace{\nu \int_{\mathcal{O}}(\Delta\v) \uprho^2\left(\frac{|x|^2}{k^2}\right) \v \d x}_{E_1(k,t)}-\underbrace{e^{z(\theta_{t}\omega)}b\left(\v,\v,\uprho^2\left(\frac{|x|^2}{k^2}\right)\v\right)}_{E_2(k,t)}\nonumber\\&\quad-\underbrace{e^{-z(\theta_{t}\omega)}\int_{\mathcal{O}}(\nabla p)\uprho^2\left(\frac{|x|^2}{k^2}\right)\v\d x}_{E_3(k,t)}+ \underbrace{e^{-z(\theta_{t}\omega)} \int_{\mathcal{O}}\f\uprho^2\left(\frac{|x|^2}{k^2}\right)\v\d x}_{E_4(k,t)} \nonumber\\&\quad+\sigma z(\theta_{t}\omega)\int_{\mathcal{O}}\uprho^2\left(\frac{|x|^2}{k^2}\right)|\v|^2\d x.
	\end{align}
	We estimate each term on the right hand side of \eqref{ep1}. Integration by parts, divergence free condition of $\v(\cdot)$,  \eqref{poin}, \eqref{p-value-N-M}, Gagliardo-Nirenberg's and Young's inequalities provide (see \eqref{ep2-N}-\eqref{ep4-N} for detailed calculations)
	\begin{align}
		E_1(k,t)&\leq-\frac{3\nu\lambda}{4} \int_{\mathcal{O}}\left|\left(\uprho\left(\frac{|x|^2}{k^2}\right) \v\right)\right|^2\d x+\frac{C}{k}\|\v\|^2_{\V},
\\
		\left|E_2(k,t)\right|&\leq\frac{C}{k}\left[e^{2\left|z(\theta_{t}\omega)\right|}\|\v\|^4_{\H}+\|\v\|^2_{\V}\right] ,
 \\
		\left|E_3(k,t)\right|&\leq \frac{C}{k}\bigg[e^{2\left|z(\theta_{t}\omega)\right|}\|\v\|^4_{\H}+\|\v\|^2_{\V}\bigg],\\
		\left|E_4(k,t)\right|&\leq \frac{\nu\lambda}{4} \int_{\mathcal{O}}\uprho\left(\frac{|x|^2}{k^2}\right)|\v|^2\d x +\frac{e^{2\left|z(\theta_{t}\omega)\right|}}{\nu\lambda} \int_{\mathcal{O}}\uprho\left(\frac{|x|^2}{k^2}\right)|\f(x)|^2\d x.\label{ep4}
	\end{align}
	Combining \eqref{ep1}-\eqref{ep4}, we get
	\begin{align}\label{ep5}
		&	\frac{\d}{\d t} \|\v\|^2_{\mathbb{L}^2(\mathcal{O}_k^{c})}+ \left(\nu\lambda-2\sigma z(\theta_{t}\omega)\right) \|\v\|^2_{\mathbb{L}^2(\mathcal{O}_k^c)} \nonumber\\ &\leq\frac{C}{k} \left[e^{2\left|z(\theta_{t}\omega)\right|}\|\v\|^4_{\H}+\|\v\|^2_{\V}\right]+\frac{2e^{2\left|z(\theta_{t}\omega)\right|}}{\nu\lambda} \int_{\mathcal{O}\cap\{|x|\geq k\}}|\f(x)|^2\d x.
	\end{align}
	Applying variation of constants formula to the equation \eqref{ep5} on $(s-t,s)$ and replacing $\omega$ by $\theta_{-s}\omega$, we find  for $s\leq\tau, t\geq 0$ and $\omega\in\Omega$,
	\begin{align}\label{ep6}
		&\|\v(s,s-t,\theta_{-s}\omega,\v_{0})\|^2_{\mathbb{L}^2(\mathcal{O}_k^{c})} \nonumber\\& \leq e^{-\alpha t+2\sigma\int_{-t}^{0}z(\theta_{\upeta}\omega)\d\upeta}\|\v_{0}\|^2_{\H}+\frac{C}{k}\bigg[\underbrace{\int_{s-t}^{s}e^{\nu\lambda(\uprho-s)-2\sigma\int^{\uprho}_{s}z(\theta_{\upeta-s}\omega)\d\upeta}\|\v(\uprho,s-t,\theta_{-s}\omega,\v_{0})\|^2_{\V}\d\uprho}_{:=\widehat{E}_1(t)}\nonumber\\&\quad+\underbrace{\int_{s-t}^{s}e^{2\left|z(\theta_{\uprho-s}\omega)\right|+\nu\lambda(\uprho-s)-2\sigma\int^{\uprho}_{s}z(\theta_{\upeta-s}\omega)\d\upeta}\|\v(\uprho,s-t,\theta_{-s}\omega,\v_{0})\|^{4}_{\H}\d\uprho}_{:=\widehat{E}_2(t)}\bigg]\nonumber\\&\quad+\underbrace{C\int_{s-t}^{s}e^{2|z(\theta_{\uprho-s}\omega)|+\nu\lambda(\uprho-s)-2\sigma\int^{\uprho}_{s}z(\theta_{\upeta-s}\omega)\d\upeta} \int_{\mathcal{O}\cap\{|x|\geq k\}}|\f(x,\xi)|^2\d x\d\uprho}_{:=\widehat{E}_3(k,t)}.
	\end{align}
From \eqref{AB3}, we deduce
\begin{align}\label{ep7}
	\widehat{E}_2(t)&\leq C\int_{s-t}^{s}e^{2\left|z(\theta_{\uprho-s}\omega)\right|+\nu\lambda(\uprho-s)+2\sigma\int_{\uprho-s}^{0}z(\theta_{\upeta}\omega)\d\upeta}\bigg[e^{-2\nu\lambda(\uprho-s+t)+4\sigma\int_{-t}^{\uprho-s}z(\theta_{\upeta}\omega)\d\upeta}\|\v_{0}\|^4_{\H} \nonumber\\&\qquad+ \bigg(\int\limits_{-t}^{\uprho-s}e^{\nu\lambda(\uprho_1+s-\uprho)+2|z(\theta_{\uprho_1}\omega)|+2\sigma\int_{\uprho_1}^{\uprho-s}z(\theta_{\upeta}\omega)\d\upeta}\|\f(\uprho_1+s)\|^2_{\H}\d\uprho_1\bigg)^2\bigg]\d\uprho\nonumber\\&\leq C\int_{-\infty}^{0}e^{2\left|z(\theta_{\uprho}\omega)\right|+\frac{\nu\lambda}{4}\uprho-2\sigma\int_{\uprho}^{0}z(\theta_{\upeta}\omega)\d\upeta}\d\uprho\cdot e^{-\frac{3\nu\lambda}{4}t+4\sigma\int_{-t}^{0}z(\theta_{\upeta}\omega)\d\upeta}\|\v_{0}\|^4_{\H}\nonumber\\&\qquad+ C\int_{-\infty}^{0}e^{2\left|z(\theta_{\uprho}\omega)\right|+\frac{\nu\lambda}{3}\uprho-2\sigma\int_{\uprho}^{0}z(\theta_{\upeta}\omega)\d\upeta}\d\uprho\nonumber\\&\qquad\times\bigg(\int_{-\infty}^{0}e^{\frac{\nu\lambda}{3}\uprho_1+2|z(\theta_{\uprho_1}\omega)|+2\sigma\int_{\uprho_1}^{0}z(\theta_{\upeta}\omega)\d\upeta}\|\f(\uprho_1+s)\|^2_{\H}\d\uprho_1\bigg)^2\nonumber\\&:=\widehat{E}_{21}(t)+\widehat{E}_{22}(t).
\end{align}
Combining \eqref{ep6} and \eqref{ep7}, we arrive at
\begin{align}
	&\|\v(s,s-t,\theta_{-s}\omega,\v_{0})\|^2_{\mathbb{L}^2(\mathcal{O}_k^{c})} \nonumber\\& \leq e^{-\alpha t+2\sigma\int_{-t}^{0}z(\theta_{\upeta}\omega)\d\upeta}\|\v_{0}\|^2_{\H}+\frac{C}{k}\bigg[\widehat{E}_1(t)+\widehat{E}_{21}(t)+\widehat{E}_{22}(t)\bigg]+\widehat{E}_{3}(k,t).
\end{align}
	Now using \eqref{f3-N}, \eqref{Z3}, the definition of backward temperedness \eqref{BackTem} and Lemma \ref{Absorbing}, one can immediately complete the proof.
\end{proof}

The following Lemma provides the backward flattening estimates for the solution of the system \eqref{2-M}.

\begin{lemma}\label{Flattening}
	Suppose that Hypothesis \ref{Hypo_f-N} is satisfied. Let $(\tau,\omega,D)\in\R\times\Omega\times\mathfrak{D}$, $k\geq1$ be fixed, $\varrho_k$ is given by \eqref{varrho_k} and $\mathrm{P}_i$ is the same as in \eqref{DirectProd}. Then
	\begin{align}\label{FL-P-M}
		\lim_{i,t\to+\infty}\sup_{s\leq \tau}\sup_{\v_{0}\in D(s-t,\theta_{-t}\omega)}\|(\I-\P_{i})\bar{\v}(s,s-t,\theta_{-s},\bar{\v}_{0,2})\|^2_{\L^2(\mathcal{O}_{2k})}=0,
	\end{align}
	where $\bar{\v}=\varrho_k\v$ and $\bar{\v}_{0,2}=(\I-\P_{i})(\varrho_k\v_{0})$.
\end{lemma}
\begin{proof}
	Multiplying by $\varrho_k$ in the first equation of \eqref{2-M}, we rewrite the equation as:
	\begin{align}\label{FL1-M}
		&\frac{\d\bar{\v}}{\d t}-\nu\Delta\bar{\v}+e^{z(\theta_{t}\omega)}\varrho_k(\v.\nabla)\v+e^{-z(\theta_{t}\omega)}\varrho_k\nabla p\nonumber\\&=e^{-z(\theta_{t}\omega)}\varrho_k\f +\sigma z(\theta_t\omega)\bar{\v}-\nu\v\Delta\varrho_k-2\nu\nabla\varrho_k\cdot\nabla\v.
	\end{align}
Applying $(\I-\P_i)$ to the equation  \eqref{FL1-M} and taking the inner product of the resulting equation with $\bar{\v}_{i,2}$ in $\L^2(\mathcal{O}_{\sqrt{2}k})$, we get
\begin{align}\label{FL2-M}
	&\frac{1}{2}\frac{\d}{\d t}\|\bar{\v}_{i,2}\|^2_{\L^2(\mathcal{O}_{\sqrt{2}k})} +\nu\|\nabla\bar{\v}_{i,2}\|^2_{\L^2(\mathcal{O}_{\sqrt{2}k})}-\sigma z(\theta_{t}\omega)\|\bar{\v}_{i,2}\|^2_{\L^2(\mathcal{O}_{\sqrt{2}k})}\nonumber\\&=-\underbrace{e^{z(\theta_{t}\omega)}\sum_{q,m=1}^{2}\int_{\mathcal{O}_{\sqrt{2}k}}\left(\I-\P_i\right)\bigg[v_{q}\frac{\partial v_{m}}{\partial x_q}\left\{\varrho_k(x)\right\}^2v_{m}\bigg]\d x}_{:=\widehat{J}_1}\nonumber\\&\quad-\underbrace{\left\{\nu\big(\v\Delta\varrho_k,\bar{\v}_{i,2}\big)+2\nu\big(\nabla\varrho_k\cdot\nabla\v,\bar{\v}_{i,2}\big)-\big(e^{-z(\theta_{t}\omega)}\varrho_k\f,\bar{\v}_{i,2}\big)\right\}}_{:=\widehat{J}_2}-\underbrace{\big(e^{-z(\theta_{t}\omega)}\varrho_k(x)\nabla p, \bar{\v}_{i,2}\big)}_{:=\widehat{J}_3}.
\end{align}
Next, we estimate the terms on the right hand side of \eqref{FL2-M} as follows. Using integration by parts, divergence free condition of $\v(\cdot)$, \eqref{poin-i} (WLOG we assume that $\lambda_{i}\geq1$), H\"older's, Gagliardo-Nirenberg's (Theorem 1, \cite{Nirenberg}) and Young's inequalities, we arrive at (see \eqref{FL3}-\eqref{FL6})
\begin{align}
	|\widehat{J}_1|&\leq \frac{\nu}{6}\|\nabla\bar{\v}_{i,2}\|^2_{\L^2(\mathcal{O}_{\sqrt{2}k})}+C\lambda^{-1}_{i+1}e^{6|z(\theta_{t}\omega)|}\|\v\|^8_{\H}+C\lambda^{-1/2}_{i+1}\|\v\|^2_{\V},\label{FL3-M}\\
	|\widehat{J}_2|&\leq \frac{\nu}{6}\|\nabla\bar{\v}_{i,2}\|^2_{\L^2(\mathcal{O}_{\sqrt{2}k})}+ C\lambda^{-1}_{i+1}\bigg[\|\v\|^2_{\V}+e^{2|z(\theta_{t}\omega)|}\|\f\|^2_{\H}\bigg],\label{FL4-M}\\
	|\widehat{J}_3|&\leq \frac{\nu}{6}\|\nabla\bar{\v}_{i,2}\|^2_{\L^2(\mathcal{O}_{\sqrt{2}k})}+C\lambda^{-1/4}_{i+1}\|\v\|^2_{\V}+C\lambda^{-1/2}_{i+1}e^{6|z(\theta_{t}\omega)|}\|\v\|^8_{\H}+C\lambda^{-1/2}_{i+1}e^{2|z(\theta_{t}\omega)|}\|\v\|^4_{\H},\label{FL5-M}
\end{align}
where we have used \eqref{p-value-N-M} in \eqref{FL5-M} also. Now, combining \eqref{FL2-M}-\eqref{FL5-M} and using \eqref{poin} in the resulting inequality, we arrive at
\begin{align}\label{FL6-M}
	&\frac{\d}{\d t}\|\bar{\v}_{i,2}\|^2_{\L^2(\mathcal{O}_{\sqrt{2}k})} +\left(\nu\lambda-2\sigma z(\theta_{t}\omega)\right)\|\bar{\v}_{i,2}\|^2_{\L^2(\mathcal{O}_{\sqrt{2}k})} \nonumber\\&\leq \I_1(i)\|\v\|^2_{\V}+\I_2(i)e^{6|z(\theta_{t}\omega)|}\|\v\|^8_{\H}+\I_3(i)e^{2|z(\theta_{t}\omega)|}\|\v\|^4_{\H}+\I_{5}(i)e^{2|z(\theta_{t}\omega)|}\|\f\|^2_{\H},
\end{align}
where
\begin{align*}
	&\I_1(i)=C\left[\lambda^{-1/4}_{i+1}+\lambda^{-1/2}_{i+1}+\lambda_{i+1}\right], \I_2(i)=C\left[\lambda^{-1/2}_{i+1}+\lambda^{-1}_{i+1}\right], \I_{3}(i)=C\lambda^{-1/2}_{i+1} \text{ and } \I_{4}(i)=C\lambda^{-1}_{i+1}.
\end{align*}
Due to the fact that $\h\in\D(\A)$ and $\lambda_i\to+\infty$ as $i\to+\infty$, we deduce that
\begin{align}\label{i_convergence-M}
	\lim_{i\to+\infty}\I_1(i)=\lim_{i\to+\infty}\I_2(i)=\lim_{i\to+\infty}\I_3(i)=\lim_{i\to+\infty}\I_4(i)=0.
\end{align} In view of the variation of constant formula applied  to \eqref{FL7}, we find
\begin{align}\label{FL8-M}
	&\|(\I-\P_{i})\bar{\v}(s,s-t,\theta_{-s}\omega,\bar{\v}_{0,2})\|^2_{\L^2(\mathcal{O}_{\sqrt{2}k})}\nonumber\\&\leq e^{-\nu\lambda t+4{\aleph}\int^{0}_{-t}\left|z(\theta_{\upeta}\omega)\right|\d\upeta}\|(\I-\P_i)(\varrho_k\v_{0})\|^2_{\L^2(\mathcal{O}_{\sqrt{2}k})}\nonumber\\&\quad+\I_1(i)\underbrace{\int_{s-t}^{s}e^{\nu\lambda(\uprho-s)-2\sigma\int^{\uprho}_{s}z(\theta_{\upeta-s}\omega)\d\upeta}\|\v(\uprho,s-t,\theta_{-s}\omega,\v_{0})\|^2_{\V}\d\uprho}_{\widehat{L}_1(s,t)}\nonumber\\&\quad+\I_2(i)\underbrace{\int_{s-t}^{s}e^{6\left|z(\theta_{\uprho-s}\omega)\right|+\nu\lambda(\uprho-s)-2\sigma\int^{\uprho}_{s}z(\theta_{\upeta-s}\omega)\d\upeta}\|\v(\uprho,s-t,\theta_{-s}\omega,\v_{0})\|^8_{\H}\d\uprho}_{\widehat{L}_2(s,t)}\nonumber\\&\quad+\I_3(i)\underbrace{\int_{s-t}^{s}e^{2\left|z(\theta_{\uprho-s}\omega)\right|+\nu\lambda(\uprho-s)-2\sigma\int^{\uprho}_{s}z(\theta_{\upeta-s}\omega)\d\upeta}\|\v(\uprho,s-t,\theta_{-s}\omega,\v_{0})\|^4_{\H}\d\uprho}_{\widehat{L}_3(s,t)}\nonumber\\&\quad+\I_4(i)\underbrace{\int_{-t}^{0}e^{2\left|z(\theta_{\uprho-s}\omega)\right|+\nu\lambda(\uprho-s)-2\sigma\int^{\uprho}_{s}z(\theta_{\upeta-s}\omega)\d\upeta}\|\f(\uprho+s)\|^2_{\H}\d\uprho}_{\widehat{L}_4(s,t)}.
 \end{align}
It implies from \eqref{Z3}, \eqref{G3} and \eqref{AB1} that
\begin{align}\label{FL9-M}
	\sup_{s\leq \tau}\widehat{L}_1(s,t)<+\infty \ \ \text{ and } \ \ \sup_{s\leq \tau}\widehat{L}_4(s,t)<+\infty,
\end{align}
for sufficiently large $t>0$. Moreover, similar arguments as in \eqref{ep7} provide
\begin{align}\label{FL10-M}
		\sup_{s\leq \tau}\widehat{L}_2(s,t)<+\infty \ \ \text{ and } \ \ \sup_{s\leq \tau}\widehat{L}_3(s,t)<+\infty,
\end{align}
for sufficiently large $t>0$. Further, we have
\begin{align}\label{FL11-M}
	\|(\I-\P_i)(\varrho_k\v_{0})\|^2_{\L^2(\mathcal{O}_{\sqrt{2}k})}\leq C\|\v_{0}\|^2_{\H},
\end{align}
for all $\v_{0}\in D(s-t,\theta_{-t}\omega)$ and $s\leq\tau$. Now, using the definition of backward temperedness \eqref{BackTem}, \eqref{Z3}, \eqref{G3}, Lemma \ref{Absorbing}, \eqref{i_convergence-M}, and \eqref{FL9-M}-\eqref{FL11-M} in \eqref{FL8-M}, we obtain \eqref{FL-P-M}, as required.
\end{proof}

	\subsection{Proof of Theorem \ref{MT1}}\label{thm1.5}
	This subsection is devoted to the main result of this section, that is, the existence of $\mathfrak{D}$-pullback random attractors and their asymptotic autonomy for the solution of the system \eqref{SNSE} with $S(\u)=\u$. The existence of pullback random attractors for non-autonomous SNSE driven by multiplicative noise on unbounded Poincar\'e domains was established in \cite{PeriodicWang}. As the existence of a unique pullback random attractor is known for each $\tau$, one can obtain the existence of a unique random attractor for a  autonomous 2D SNSE driven by multiplicative noise on unbounded Poincar\'e domains (cf. \cite{PeriodicWang}).

	In view of Propositions \ref{Back_conver} and \ref{IRAS}, and Lemmas \ref{largeradius} and \ref{Flattening}, we can prove the Theorem \ref{MT1} by applying similar arguments as in the proof of Theorem \ref{MT1-N}, see Subsection \ref{thm1.4}.

	\medskip\noindent
{\bf Acknowledgments:} Renhai Wang was supported by  China
Postdoctoral Science Foundation under grant numbers 2020TQ0053 and 2020M680456.   Kush Kinra   would like to thank the Council of Scientific $\&$ Industrial Research (CSIR), India for financial assistance (File No. 09/143(0938)/2019-EMR-I).  M. T. Mohan would  like to thank the Department of Science and Technology (DST), Govt of India for Innovation in Science Pursuit for Inspired Research (INSPIRE) Faculty Award (IFA17-MA110).

	\medskip\noindent
{\bf Data availability:} Data sharing not applicable to this article as no datasets were generated or analysed during the current study.


\begin{thebibliography}{99}
	\bibitem{Arnold}	L. Arnold, \emph{Random Dynamical Systems}, Springer-Verlag, Berlin, Heidelberg, New York, 1998.
	\bibitem{Ball} J. M. Ball, Global attractors for damped semilinear wave equations, \emph{Discrete Contin. Dyn. Syst. Ser. B}, \textbf{10} (2004), 31-52.


\bibitem{Ball1997JNS} J. M. Ball, Continuity properties and global attractors of generalized semiflows and the
Navier-Stokes equations, \emph{J. Nonl. Sci.} \textbf{7} (1997) 475-502.


	\bibitem{BCF}  	Z. Brze\'zniak, M. Capi\'nski and F. Flandoli, Pathwise global attractors for stationary random dynamical systems, \emph{Probab. Theory Related Fields}, \textbf{95} (1993), 87-102.
	\bibitem{BCLLLR} Z. Brz\'ezniak, T. Caraballo, J. A. Langa, Y. Li, G. Lukaszewicz and J. Real, Random attractors for stochastic 2D Navier-Stokes equations in some unbounded domains, \emph{J. Differential Equations}, \textbf{255} (2013), 3897--3919.
	\bibitem{BL} Z. Brz\'ezniak and Y. Li, Asymptotic compactness and absorbing sets for 2D stochastic Navier-Stokes equations in some unbounded domains, \emph{Trans. Amer. Math. Soc.}, \textbf{358} (12) (2006) 5587--5629.

\bibitem{chenp} P. Chen, B. Wang, R. Wang, X. Zhang, Multivalued random dynamics of Benjamin-Bona-Mahony equations driven by nonlinear colored noise on unbounded domains, \emph{Math. Ann.} (2022), https://doi.org/10.1007/s00208-022-02400-0.






	\bibitem{CGSV} T. Caraballo, M. J. Garrido-Atienza, B. Schmalfuss and J. Valero, Asymptotic behaviour of a stochastic semilinear dissipative functional equation without uniqueness of solutions, \emph{Discrete Contin. Dyn. Syst. Ser. B}, \textbf{14} (2) (2010), 439--455.
	\bibitem{CGTW} T. Caraballo, B. Guo, N. Tuan and R. Wang, Asymptotically autonomous robustness of random attractors for a class of weakly dissipative stochastic wave equations on unbounded domains, \emph{Proc. Roy. Soc. Edinburgh Sect. A}, \textbf{151} (6) (2021), 1700-1730, \url{doi:10.1017/prm.2020.77}.
	\bibitem{CLR} T. Caraballo, G. Lukaszewicz and J. Real, Pullback attractors for asymptotically compact non-autonomous dynamical systems, \emph{Nonlinear Anal.} \textbf{64}(3) (2006), 484-498.
	\bibitem{CLR1} T. Caraballo, G. Lukaszewicz and J. Real, Pullback attractors for non-autonomous 2D-Navier–Stokes equations in some unbounded domains, \emph{C. R. Acad. Sci. Paris, Ser. I}, \textbf{342} (4) (2006), 263-268.

\bibitem{CLR2} A. Carvalho, J. A. Langa and J. Robinson, \emph{Attractors for Infinite-dimensional Non-autonomous Dynamical Systems}, Netherlands: Springer New York, 2013.

	\bibitem{CV2} V. V. Chepyzhov and M. I. Vishik, \emph{Attractors for Equations of Mathematical Physics}, American Mathematical Society Colloquium Publications, \textbf{49} American Mathematical Society, Providence, RI, 2002.
	\bibitem{CDF}	H. Crauel, A. Debussche, F. Flandoli, Random attractors, \emph{J. Dynam. Differential Equations} \textbf{9}(2)(1995), 307-341.
	\bibitem{CF}	H. Crauel and F. Flandoli, Attractors for random dynamical systems, \emph{Probab. Theory Related Fields}, \textbf{100} (1994), 365-393.


\bibitem{Chenzhang2} Z. Chen, B. Wang, Weak mean attractors and invariant measures for stochastic Schr\"{o}dinger delay lattice systems, \emph{J. Dynam. Differential Equations}, 2022, https://doi.org/10.1007/s10884-021-10085-3.




\bibitem{Chueshov1}
I. Chueshov, I. Lasiecka, \emph{Von Karman Evolution Equations: Well-posedness and Long Time Dynamics}, Springer Monographs in Mathematics, New York, 2010.

 \bibitem{Chueshov2} I, Chueshov, I. Lasiecka,  \emph{Long-Time Behavior of Second Order Evolution Equations with Nonlinear Damping}, Memoirs of the American Mathematical Society, \textbf{195}, 2008.





	\bibitem{FAN} X. Fan, Attractors for a damped stochastic wave equation of the sine-Gordon type with sublinear multiplicative noise, \emph{Stoch. Anal. Appl.}, \textbf{24} (2006), 767–793.
	\bibitem{FY}	X. Feng and B. You, Random attractors for the two-dimensional stochastic g-Navier-Stokes equations, \emph{Stochastics}, \textbf{92}(4) (2020), 613-626.
	\bibitem{FSX} B. Fernando, S. S. Sritharan and M. Xu, A simple proof of global solvability for 2-D Navier-Stokes equations in unbounded domains, \emph{Differential Integral Equations}, \textbf{23} (3-4), (2010), 223–235.
	\bibitem{FS} F. Flandoli and B. Schmalfuss, Random attractors for the 3D stochastic Navier-Stokes equation with multiplicative noise, \emph{Stoch. Stoch. Rep.}, \textbf{59}(1-2), 1996, 21–45.
	\bibitem{FMRT} C. Foias, O. Manley, R. Rosa and R. Temam, \emph{Navier-Stokes Equations and Turbulence}, Cambridge University Press, 2008.
	\bibitem{GZ} N. Glatt-Holtz and M. Ziane, Strong pathwise solutions of the stochastic Navier-Stokes system, \emph{Adv. Differential Equations}, \textbf{14} (5/6) (2009), 567-600.
	\bibitem{GGW}  A. Gu, B. Guo and B. Wang, Long term behavior of random Navier-Stokes equations driven by colored noise, \emph{Discrete Contin. Dyn. Syst. Ser. B}, \textbf{25}(7) (2020), 2495-2532.
	\bibitem{GLW} A. Gu, K. Lu and B. Wang, Asymptotic behavior of random Navier-Stokes equations driven by Wong-Zakai approximations, \emph{Discrete Contin. Dyn. Syst. Ser. B}, \textbf{39}(1) (2019), 185-218.
	\bibitem{Heywood} J. G. Heywood, The Navier-Stokes Equations: on the existence, regularity and decay of solutions, \emph{Ind. Univ. Math. J.}, \textbf{29} (1980), 639-681.



\bibitem{Hale1}J.K. Hale, G. Raugel, A damped hyperbolic equation on thin domains, Trans. Amer. Math. Soc. 329 (1992) 185-219.

\bibitem{Hale2}J.K. Hale, G. Raugel,
 Reaction-diffusion equations on thin
domains, J. Math. Pures Appl. 71 (1992) 33–95.



	\bibitem{KM3} K. Kinra and M. T. Mohan, Large time behavior of the deterministic and stochastic 3D convective Brinkman-Forchheimer equations in periodic domains, \emph{J. Dynam. Differential Equations}, (2021), pp. 1-42.
	\bibitem{KM4} K. Kinra and M. T. Mohan, Weak pullback mean random attractors for the stochastic convective Brinkman-Forchheimer equations and locally monotone stochastic partial differential equations, \emph{Infin. Dimens. Anal. Quantum Probab. Relat. Top.}, \textbf{25}(01) (2022), 2250005.
	\bibitem{KM7} K. Kinra and M. T. Mohan, Long term behavior of 2D and 3D non-autonomous random convective Brinkman-Forchheimer equations driven by colored noise, \emph{Submitted}, \url{https://arxiv.org/pdf/2107.08890.pdf}.
	\bibitem{KRM} K. Kinra, R. Wang and M. T. Mohan, Asymptotic Autonomy of Random Attractors in Regular Spaces for Non-autonomous Stochastic Navier-Stokes Equations, \emph{Submitted}, \url{https://arxiv.org/pdf/2205.02099.pdf}.
	\bibitem{Kuratowski} K. Kuratowski, Sur les espaces complets, \emph{Fund. Math.}, \textbf{15} (1930) 301–309.
	\bibitem{OAL}	O. A. Ladyzhenskaya, \emph{The Mathematical Theory of Viscous Incompressible Flow}, Gordon and Breach, New York, 1969.

	\bibitem{LGL} Y. Li, A. Gu and J. Li, Existence and continuity of bi-spatial random attractors and application to stochastic semilinear Laplacian equations, \emph{J. Dynam. Differential Equations}, \textbf{258}(2) (2015), 504-534.

	\bibitem{YR} Y. Li and R. Wang, Asymptotic autonomy of random attractors for BBM equations with Laplace-multiplier noise, \emph{J. Appl. Anal. Comput.}, \textbf{10}(4) (2020), 1199-1222, DOI: \url{10.11948/20180145}.
	
	\bibitem{LX} F. Li, D. Xu and J. Yu, Regular measurable backward compact random attractor for $\boldsymbol{g}$-Navier-Stokes equation, \emph{Commun. Pure Appl. Anal.}, \textbf{19}(6) (2020), p.3137.
	\bibitem{LX1}F. Li and D. Xu, Asymptotically autonomous dynamics for non-autonomous stochastic $\boldsymbol{g}$-Navier-Stokes equation with additive noise, \emph{Discrete Contin. Dyn. Syst. Ser. B}, (2022).
	
	\bibitem{LG} H. Liu and H. Gao, Ergodicity and dynamics for the stochastic 3D Navier-Stokes equations with damping, \emph{Commun. Math. Sci.}, \textbf{16}(1) (2018), 97-122.
	\bibitem{Malkowsky} E. Malkowsky, Measures of noncompactness and some applications, \emph{Contemp. Anal. Appl. Math.}, \textbf{1}(1) (2013), 2-19.
	\bibitem{MS} J. L. Menaldi and S. S. Sritharan, Stochastic 2-D Navier-Stokes equation, \emph{Appl. Math. Optim.}, \textbf{46} (2002), 31–53.
\bibitem{Nirenberg} 	L. Nirenberg,  On elliptic partial differential equations, \emph{Ann. Scuola Norm. Sup. Pisa} {\bf 3} 13 (1959), 115--162.

\bibitem{Rakocevic} V. Rako$\check{c}$evi\'c, Measures of noncompactness and some applications, \emph{Filomat} (1998), 87-120.

\bibitem{Robinson2} J. C. Robinson, \emph{Infinite-Dimensional Dynamical Systems, An Introduction to Dissipative Parabolic PDEs and the Theory of Global Attractors}, Cambridge Texts in Applied Mathematics, 2001.


\bibitem{Robinson1}J. C. Robinson, Dimensions, Embeddings and Attractors, Cambridge Tracts in Mathematics, Cambridge University Press, Cambridge, 2011.


\bibitem{SS} K. Sakthivel, S. S. Sritharan, Martingale solutions for stochastic Navier-Stokes equations driven by Lévy noise, \emph{Evol. Equ. Control Theory}, \textbf{1} (2) (2012), 355--392, doi: 10.3934/eect.2012.1.355.
\bibitem{Schmalfussr} B. Schmalfu{\ss}, Backward cocycle and attractors of stochastic differential equations, In  International Seminar on Applied Mathematics Nonlinear Dynamics: Attractor Approximation and Global Behavior (V. Reitmann, T. Riedrich, and N. Koksch, eds.), Technische Universit\"{a}t Dresden, 1992, 185-192.
\bibitem{R.Temam}	R. Temam, \emph{Infinite-Dimensional Dynamical Systems in Mechanics and Physics,} vol. 68, Applied Mathematical Sciences,	Springer, 1988.
\bibitem{Temam1}	R. Temam, \emph{Navier-Stokes Equations and Nonlinear Functional Analysis}, Second Edition, CBMS-NSF Regional Conference Series in Applied Mathematics, 1995.
\bibitem{UTE-Wang} B. Wang, Attractors for reaction-diffusion equations in unbounded domains, \emph{Physica D}, \textbf{128} (1999), 41–52.
\bibitem{PeriodicWang} B. Wang, Periodic random attractors for stochastic Navier-Stokes equations on unbounded domain, \emph{Electron. J. Differential Equations}, \textbf{2012} (59) (2012), 1-18.
\bibitem{SandN_Wang} B. Wang, Sufficient and necessary criteria for existence of pullback attractors for non-compact random dynamical systems, \emph{J. Differential Equations}, \textbf{253} (5) (2012), 1544-1583.
\bibitem{Wang}	B. Wang, Weak pullback attractors for mean random dynamical systems in Bochner spaces, \emph{J. Dynam. Differential Equations}, {\bf 31} (2019), 2177-2204.
\bibitem{Wang1}	B. Wang, Weak pullback attractors for stochastic Navier-Stokes equations with nonlinear diffusion terms, \emph{Proc. Amer. Math. Soc.}, \textbf{147}(4) (2019), 1627-1638.



\bibitem{Wang2011Tran} B. Wang, Asymptotic behavior of stochastic wave equations with critical exponents on $\mathbb{R}^{3}$, Tran. Amer. Math. Soc. 363 (2011) 3639-3663.

\bibitem{wangjfa} B. Wang, Well-Posedness and Long Term Behavior of Supercritical Wave Equations Driven by Nonlinear Colored Noise on $\mathbb{R}^n$, Journal of Functional Analysis, 2022, https://doi.org/10.1016/j.jfa.2022.109498.




\bibitem{rwang1}R. Wang, Y. Li and B. Wang, Random dynamics of fractional nonclassical diffusion equations driven by colored noise, \emph{Discrete Contin. Dyn. Syst.}, \textbf{39} (2019), 4091-4126.

\bibitem{rwang2} R. Wang, L. Shi and  B. Wang, Asymptotic behavior of fractional nonclassical diffusion equations driven by nonlinear colored noise on $\mathbb{R}^N$, \emph{Nonlinearity}, \textbf{32} (2019), 4524-4556.




\bibitem{Wang4}R. Wang and B. Wang, Random dynamics of $p$-Laplacian lattice systems driven by infinite-dimensional nonlinear noise, \emph{Stochastic Process. Appl.}, \textbf{130}, (2020) 7431-7462.



\bibitem{Wang8} R. Wang, B. Guo, B. Wang, Well-posedness and dynamics of fractional FitzHugh-Nagumo systems on $\mathbb{R}^N$ driven by nonlinear noise, \emph{Sci. China Math.}, \textbf{64} (2021), 2395-2436.

\bibitem{Wang9} R. Wang, Long-time dynamics of stochastic lattice plate equations with nonlinear noise and damping, \emph{J. Dynam. Differential Equations}, \textbf{33}(2) (2021) 767-803.


\bibitem{WL} S. Wang and Y. Li, Longtime robustness of pullback random attractors for stochastic magneto-hydrodynamics equations, \emph{Physica D}, \textbf{382} (2018), 46–57.
\bibitem{XL} D. Xu and F. Li, Asymptotically autonomous dynamics for non-autonomous stochastic 2D $\boldsymbol{g}$-Navier–Stokes equation in regular spaces, \emph{J. Math. Phys.}, \textbf{63}(5) (2022): 052701.
\bibitem{YGL} J. Yin, A. Gu and Y. Li, Backwards compact attractors for non-autonomous damped 3D Navier-Stokes equations, \emph{Dyn. Partial Differ. Equ.}, \textbf{14} (2017), 201–218.
\bibitem{ZL} Q. Zhang and Y. Li, Regular attractors of asymptotically autonomous stochastic 3D Brinkman-Forchheimer equations with delays, \emph{Commun. Pure Appl. Anal.}, \textbf{20}(10) (2021), p.3515.

\end{thebibliography}
\end{document}